\documentclass[final]{siamltex}

\usepackage[top=2.5cm,bottom=2.5cm,right=2.5cm,left=2.5cm]{geometry}

\usepackage{endnotes}
\let\footnote=\endnote

\def\g{\gamma}
\def\t{\theta}
\def\o{\omega}

\newcommand{\tsum}{\textstyle{\sum}}

\def\sl{\textstyle{\sum}_{\ell=1}^k}

\def\R{\mathbb{R}}

\def\la{\langle}
\def\ra{\rangle}

\def\Eb{\mathbb{E}}
\def\lc{\left\lceil}   
\def\rc{\right\rceil}

\def\eqnok#1{(\ref{#1})}
\usepackage{epsfig}
\usepackage{calc}
\usepackage{amstext}
\usepackage{amsmath}
\usepackage{multicol}
\usepackage{amsfonts}
\usepackage{amssymb}
\usepackage{mathrsfs}
\usepackage{bm}
\usepackage{xcolor}
\usepackage{comment}
\usepackage[scientific-notation=true]{siunitx}

\newcommand{\T}{\mathrm{T}}

\newcommand{\gap}{\mathrm{gap}}
\newcommand{\res}{\mathrm{res}}
\newcommand{\erf}{\mathrm{erf}}

\newcommand{\bbr}{\Bbb{R}}
\newcommand{\beq}{\begin{equation}}
\newcommand{\eeq}{\end{equation}}
\newcommand{\beqa}{\begin{eqnarray}}
\newcommand{\eeqa}{\end{eqnarray}}
\newcommand{\beqas}{\begin{eqnarray*}}
\newcommand{\eeqas}{\end{eqnarray*}}

\usepackage{algorithm,algorithmic}
\usepackage{bbm}
\def\argmin{{\rm argmin}}
\def\argmax{{\rm argmax}}

\def\bbe{{\mathbb{E}}}
\def\vgap{\vspace*{.1in}}
\newcommand{\nn}{\nonumber}

\title{Simple and optimal methods for stochastic variational inequalities, I: operator extrapolation
\thanks{
This research was partially supported by the ARO grant W911NF-18-1-0223 and ONR grant N00014-20-1-2089.
Coauthors of this paper are listed according to the alphabetic order.
}}
\author{
Georgios Kotsalis
	\thanks{H. Milton Stewart School of Industrial and Systems Engineering, Georgia Institute of Technology, Atlanta, GA, 30332 .
		(email: {\tt gkotsalis3@gatech.edu}).}
	\and
Guanghui Lan
	\thanks{H. Milton Stewart School of Industrial and Systems Engineering, Georgia Institute of Technology, Atlanta, GA, 30332 .
		(email: {\tt george.lan@isye.gatech.edu}).}
	\and
Tianjiao Li
	\thanks{H. Milton Stewart School of Industrial and Systems Engineering, Georgia Institute of Technology, Atlanta, GA, 30332 .
		(email: {\tt tli432@gatech.edu}).}
}

\date{\today}

\begin{document} 

\maketitle

\begin{abstract}
In this paper we first present a novel operator extrapolation (OE) method
for solving deterministic variational inequality (VI) problems. Similar to the gradient (operator) projection method,
OE updates one single search sequence by solving a single projection subproblem in each iteration.
We show that OE can achieve the optimal rate of convergence for solving a variety of
VI problems in a much simpler way than existing approaches.
We then introduce the stochastic operator extrapolation (SOE) method and establish 
its optimal convergence behavior for solving different stochastic VI problems.
In particular, SOE achieves the optimal complexity for solving a fundamental problem, i.e., stochastic smooth and strongly
monotone VI, for the first time in the literature. We also present
a stochastic block operator extrapolations (SBOE) method to further reduce the iteration cost for the OE method applied to large-scale
deterministic VIs with a certain block structure.
Numerical experiments
have been conducted to demonstrate the potential advantages of the proposed algorithms.
In fact, all these algorithms are applied to solve generalized monotone variational inequality (GMVI) problems
whose operator is not necessarily monotone.
We will also discuss optimal OE-based policy evaluation methods for reinforcement learning in a companion paper.\\

 \end{abstract}

\section{Introduction}   
\textcolor{black}{A unified framework to solve optimization, fixed-point equations, equilibrium and complementarity problems is provided by the concept of a variational inequalitiy (VI) }(see  \cite{FacPang03} for an extensive review and bibliography).
In this paper, we consider a class of generalized monotone variational inequality (GMVI) problems:
\begin{equation}
\label{VIP}
\text{Find}~ x^* \in X: ~~~~ \langle F(x^*) , x - x^* \rangle \geq 0, ~~~ \forall x \in X,
\end{equation}
where $X \subseteq \mathbb{R}^n$ is a nonempty closed convex set, and $F : X \rightarrow \mathbb{R}^n$ is an $L$-Lipschitz continuous map, i.e. for some $L > 0 $,
\begin{equation}
\label{Lipschitz}
\| F(x_1) - F(x_2) \|_*   \leq L \| x_1 - x_2 \|,~~~~ \forall x_1, x_2 \in X.
\end{equation}
VIs satisfying \eqnok{Lipschitz} are often said to be smooth.
In addition, we assume that $F$ satisfies a generalized monotonicity condition 
\begin{equation} \label{G_monotone0}
\langle F(x), x - x^* \rangle \ge \mu \|x - x^*\|^2, \ \ \forall x \in X
\end{equation}
for some $\mu \ge 0$. Throughout this paper we assume the existence of the solution $x^*$ to problem \eqnok{VIP}-\eqnok{G_monotone0}.

Clearly, condition \eqref{G_monotone0} holds if $F$ is monotone, 
i.e., there exists $ \mu \geq 0 $ such that 
\begin{equation}
\label{strongly_monotone}
\langle F(x_1) - F(x_2), x_1 - x_2 \rangle \geq \mu \|x_1 - x_2\|^2,~~~~ \forall x_1, x_2 \in X.
\end{equation}
In particular, $F$ is strongly monotone if $\mu > 0$ in \eqnok{strongly_monotone}. 
However, condition \eqnok{G_monotone0} does not necessarily imply the monotonicity of $F$. For example, pseudo-monotone
VIs satisfy \eqnok{G_monotone0}, but not \eqnok{strongly_monotone}~\cite{dang2015convergence,LanBook2020}. 
A notion related to strong monotonicity
is the so-called weak sharpness condition, i.e.,
$\exists \mu > 0 $ s.t. 
\beq \label{weak_sharpness}
\langle F(x^*), x - x^* \rangle \geq \mu \|x - x^*\|^2,~~~\forall x \in X.
\eeq
Note that if $F$ is monotone and satisfies \eqnok{weak_sharpness}, then \eqref{G_monotone0} must hold.
Therefore,  GMVI covers many VI problems that have been studied in the literature. 
As a special case of GMVI, we call problem~\eqnok{VIP}-\eqnok{G_monotone0} generalized strongly monotone VI (GSMVI) if $\mu > 0$ in \eqnok{G_monotone0}.
In this paper, we consider both deterministic GMVI with exact information about the operator $F$, and stochastic GMVI with 
only unbiased estimators of $F$ obtained through a stochastic oracle. 

VI has attracted much interest recently in machine learning, statistics and artificial intelligence in addition to 
some more traditional applications, e.g., in transportation (see, e.g., \cite{Dafermos80} and Section~\ref{transportation}). Specifically,
GMVIs have been used to solve a class of minimax problems inspired by
the Generative Adversary Networks (e.g., \cite{LinLiuRafiqueYang2018}).
More recently, Juditsky and Nemirovski developed some interesting applications 
of stochastic strongly monotone VIs in signal estimation using generalized linear models (Chapter 5.2 of  \cite{ArkadiBook2020},
and Section~\ref{arkadi examples}). Our study has also been
motivated by the emergent application of stochastic strongly monotone VIs for policy evaluation
in reinforcement learning (RL)~\cite{Bertsekas2011}. For a given linear operator $T: \bbr^n \to \bbr^n$,
the basic policy evaluation problem can be formulated as a fixed-point equation 
\[
\text{Find}~ x^* \in X: ~~~~ x^* = T(x^*),
\]
which is a special case of \eqref{VIP}
with $X = \bbr^n$ and $F(x) = x - T(x)$.
Observe that usually we do not have unbiased estimators of $F$ for the above VIs due to the existence of Markovian noise in RL.
We will study this more challenging problem in a companion paper after 
developing some basic results in this work.

VI has been the focus of  many algorithmic studies due to their relevance in practice.
Classical algorithms for VI include, but not limited to, the gradient projection
method (e.g., \cite{Bert99}, \cite{Sibony70}), Korpelevich’s extragradient method \cite{korpelevich1976extragradient}, and the proximal point algorithm (e.g. \cite{rockafellar1976monotone}).
While earlier studies  focused on the asymptotic convergence behavior of different algorithms,
much recent
research effort has been devoted to algorithms exhibiting strong performance guarantees in a finite number
of iterations (a.k.a., iteration complexity).
More specifically, Nemirovski 
in \cite{Nem05-1} presented a mirror-prox method by properly modifying the extragradient algorithm \cite{korpelevich1983extrapolation} and show that
it can achieve an  ${\cal O} \{1/\epsilon\}$ complexity bound for solving smooth monotone VI problems. Here, $\epsilon >0 $ denotes the target accuracy in terms  
of weak gap, i.e., $\gap(\bar{x}) := \max_{x \in X} \langle F(x),  \bar{x} -x \rangle$.
This bound significantly improves the ${\cal O} \{1/\epsilon^2\}$  
bound for solving VIs with bounded operators (i.e.,
nonsmooth VIs) (e.g., \cite{BenNem05-1}).
Nemirovski's mirror-prox method has inspired many studies for solving VI problems under either deterministic  (see, e.g., \cite{auslender2005interior,Nest07-2,NesterovScimali06,MonSva10-1,dang2015convergence, Malisky15})
or stochastic (see, e.g., \cite{NJLS09-1,juditsky2011solving,CheLanOu14-1,yousefian2017smoothing,iusem2017extragradient,Shanbhag2019}) settings.

In spite of these efforts,
there remain a few significant issues on the development of efficient solution methods for VIs.
Firstly, there does not exist an optimal
method for solving stochastic smooth and strongly monotone VIs,
even though many VIs arising from different applications are given in this form~\cite{Bertsekas2011,ArkadiBook2020}. In the deterministic setting,
Nesterov and Scimali~\cite{NesterovScimali06} show that the dual extrapolation method~\cite{Nest07-2}
can achieve an optimal ${\cal O}\{(L/\mu) \log (1/\epsilon)\}$ complexity bound in term of the
distance to the optimal solution, which significantly improves the 
${\cal O} \{(L/\mu)^2 \log (1/\epsilon) \}$ complexity bound of the gradient (operator) projection method.
However, only an ${\cal O}(1/\epsilon^2)$ sampling complexity bound has been reported for 
solving stochastic smooth and strongly monotone VIs~\cite{Shanbhag2019}. The latter bound
is actually worse than the ${\cal O}\{1/\epsilon\}$ complexity
bound for solving stochastic nonsmooth strongly monotone VIs, which can be obtained by applying the stochastic approximation method
to VI~\cite{NJLS09-1,LanBook2020}.  However, it is well-known that these nonsmooth methods are
not optimal for smooth VIs in terms of the dependence on the condition number and
the initial distance to the optimal solution. They can not take advantage of a variety of variance reduction techniques 
available for smooth problems either. 

Secondly, most existing optimal VI algorithms, including the mirror-prox method~\cite{Nem05-1} and the dual extrapolation 
method~\cite{Nest07-2}, are somewhat complicated in the sense that they involve 
the updating of two or three intertwined sequences and the solution of two 
projection subproblems in each iteration.
More recently, Malisky~\cite{Malisky15} develops a novel variant of the mirror-prox method for smooth monotone VIs,
called projected reflected gradient method,
which only requires one projection step in each iteration. However,
this method in~\cite{Malisky15} still has the following possible limitations:  a) 
it requires $F$ to be defined over the $\bbr^n$ rather than the feasible region $X$;
b) it still needs to maintain the updating
of two sequences; and c) it \textcolor{black}{does not} easily extend for solving strongly monotone VIs.
To the best of our knowledge, there does not exist an optimal method for VIs that requires
the computation of only one sequence, similarly to the simple gradient projection
method. 

In this paper, we aim to address some of the aforementioned issues in
the design and analysis of efficient VI solution methods. Our main contributions can be briefly summarized as follows.
Firstly, we present a new and simple first-order algorithm, called {\em operator extrapolation} (OE) method, which
only requires the updating of one sequence $\{x_t\}$ for solving VI problems. Each OE iteration involves 
the evaluation of the operator $F$ once at $x_t$, and the updating from $x_t$ to $x_{t+1}$ through
only one projection subproblem:
\[
x_{k+1} = \argmin_{x\in X} ~   \gamma_t \big\langle F(x_t) + \lambda_t \big( F(x_t) - F(x_{t-1}) \big) , x \big\rangle + \tfrac{1}{2} \|x - x_t\|^2.
\]
Note that the above Euclidean projection can be generalized to the non-Euclidean setting (see Section~\ref{sec_term}).
We show that OE exhibits an ${\cal O}\{(L/\mu) \log (1/\epsilon)\}$ iteration complexity in terms of 
the distance to the optimal solution for GSMVIs, which cover smooth and strongly monotone VIs as a special case. Hence,
it achieves the optimal complexity for the latter class of VIs in a much simpler way than \cite{Nest07-2},
and the one for the more general GSMVIs for the first time in the literature.
Moreover, it exhibits an ${\cal O}\{1/\epsilon^2\}$ complexity in terms of residual (i.e., $\res(\bar x) := \min_{y \in - N_X(\bar x)} \|y - F(\bar x)\|_*$)
for solving GMVIs which are not necessarily \textcolor{black}{strongly} monotone (i.e., $\mu = 0$ in \eqnok{G_monotone0}).
In addition, we show \textcolor{black}{in the case that $X$ is bounded } that OE can achieve an optimal ${\cal O}\{1/\epsilon\}$ complexity in terms of weak gap 
for solving smooth monotone VIs similar to \cite{Nem05-1}. 

Secondly, we present a stochastic OE (SOE) method obtained by replacing the
operator $F(x_t)$ with its unbiased estimator in OE for solving stochastic VI. We show that SOE employed with different stepsize policies
can achieve either nearly optimal or optimal complexity for stochastic GSMVI. More specifically,
with only one sample of the random variables in each iteration, 
it achieves the ${\cal O} \{ L / (\mu \sqrt{\epsilon}) + \sigma^2/(\mu^2 \epsilon)\}$ 
and ${\cal O} \{ (L/\mu) \log (1/\epsilon) + \sigma^2 \log (1/\epsilon) /(\mu^2 \epsilon)\}$
nearly optimal sampling complexity bounds, respectively, by using a simple decreasing stepsize policy
and a new constant stepsize policy. By using a novel index-resetting stepsize policy, 
it achieves the optimal  ${\cal O} \{ (L/\mu) \log (1/\epsilon) + \sigma^2 /(\mu^2 \epsilon)\}$ complexity bound.
To the best of our knowledge, all these complexity bounds are new for solving stochastic GSMVIs. In fact,
they significantly improve the existing ${\cal O}\{1/\epsilon^2\}$ sampling complexity for solving stochastic smooth and strongly
monotone VIs in \cite{Shanbhag2019}. In comparison with nonsmooth stochastic VI methods, SOE can benefit from variance reduction obtained
from mini-batches which are widely used in practice. In addition, we establish the convergence of SOE
for solving stochastic GMVI and monotone VI
in terms of expected \textcolor{black}{squared} residual and weak gap, respectively.

Thirdly, we develop a stochastic block OE (SBOE) method for solving
GSMVIs whose feasible region is defined as $X := X_1 \times X_2 \times \cdots X_b$.
The SBOE method randomly updates one block of variables and hence its iteration cost is
cheaper than that of OE, especially if some recursive updating of the operator $F$ is employed.
We show that SBOE still achieves a linear rate of convergence for solving GSMVIs.  
In addition, we show that SBOE can achieve an ${\cal O}\{1/\epsilon\}$ iteration complexity in terms of expected weak gap 
for solving smooth monotone VI problems.
To the best of our knowledge,
this is the first time that the convergence rate of this type of method has been
established for solving VI problems in the literature,
while stochastic block coordinate descent methods have been intensively studied for solving
optimization problems.

Finally, we conduct numerical experiments on the proposed
algorithms, including OE, SOE and SBOE, applied to solve a class of traffic assignment problems and
signal estimation problems, and demonstrate their advantages over 
some existing VI methods, including the dual extrapolation method and 
 stochastic approximation for solving smooth and strongly monotone VIs.

This paper is organized as follows. We discuss the OE method and its convergence properties for deterministic 
VIs in Section~\ref{Sec 2}. We then present the SOE method for solving stochastic VIs 
in Section~\ref{sec_SOE}. The SBOE method for deterministic GSMVIs is discussed in Section~\ref{sec_SBOE}, and
the results from our numerical experiments are reported in Section~\ref{sec-num}. 
We complete this paper with some brief concluding remarks in Section~\ref{sec_conclusion}.

\subsection{Notation and terminology} \label{sec_term}
Let $\mathbb{R}$ denote the set of real numbers.
All vectors are viewed as column vectors, and
for a vector $x \in \mathbb{R}^d$, we use $x^{\top}$ to denote its transpose. 
The identity matrix in $\mathbb{R}^d$ is denoted by $I_d$.  \textcolor{black}{Given a norm $ \| \cdot \| $ in $\mathbb{R}^d$ the associated dual norm $ \| \cdot \|_*$ is defined as 
$
\| z \|_* = \sup \{ \langle x, z \rangle : \| x \| \leq 1 \}.
$}
For any $n \ge 1$, the set of integers $\{1,\ldots,n\}$ is denoted by $[n]$.
For any $s\in \R$,  $\lceil s \rceil$ denotes
the nearest integer to $s$ from above. 
We use $\Eb_s[X]$ to denote the expectation of a random variable $X$ on $\{i_1,\ldots, i_s\}$.    
For a given strongly convex function $\o$ with modulus $1$, we define the prox-function (or Bregman's distance) associated with $\o$ as
$V(x,y) \equiv V_{\o}(x,y): = \o(y)-\o(x)-\la \o^{\prime}(x),y-x\ra$, $\forall x,y \in X$,
where $\o^{\prime}(x)\in \partial \o(x)$  is an arbitrary subgradient of $\o$ at $x$.
Note that by the strong convexity of $\o$, we have
\beq \label{strong_convex_V}
V(x,y) \ge \tfrac{1}{2} \|x - y\|^2.
\eeq
With the definition of the Bregman's distance, we can replace the generalized
strong monotonicity assumption in \eqnok{G_monotone0} by 
\begin{equation} \label{G_monotone}
\langle F(x), x - x^* \rangle \ge 2 \mu V(x, x^*), \ \ \forall x \in X.
\end{equation}

\section{Deterministic VIs}\label{Sec 2}
We focus on the operator extrapolation (OE) method for solving deterministic VIs in this section.
\subsection{The operator extrapolation method}

As shown in Algorithm~\ref{alg:deterministic},
the basic  algorithmic scheme for the proposed OE method is conceptually simple.
It only involves a single sequence of iterates $ \{ x_t \} $, along with two sequences of 
nonnegative parameters $ \{ \gamma_t \} $ and  $ \{ \lambda_t \} $, and a prox-function $ V : X \times X \rightarrow \mathbb{R}$. 
The parameters $\{\lambda_t\}$ define the way we take extrapolation
on the operators, while the parameters $\{\gamma_t\}$ can be viewed as stepsizes.
If $V(x_t, x) = \|x - x_t\|_2^2 /2 $ and $\lambda_t = 0$, then this method
reduces to the well-known operator projection method for VI, or gradient projection method if $F$ is a gradient field.
A distinctive feature of the OE method is that, when $\lambda_t > 0$, it performs an operator 
extrapolation step given by $ F(x_t) + \lambda_t \big( F(x_t) - F(x_{t-1}))$, before the projection on $X$.
While the extrapolation of gradients and its relation with Nesterov's acceleration
had been been studied before (see Section 3 of~\cite{LanZhou2018RGEM}), the extrapolation of operators for VIs
has not been studied before in the literature to the best of our knowledge.

\begin{algorithm}[H]  \caption{The Operator Extrapolation (OE) Method}  
	\label{alg:deterministic}
	\begin{algorithmic} 
		\STATE{Let $x_0 = x_1 \in X$, and the nonnegative parameters $\{ \gamma_t\}$ and $\{\lambda_t\}$ be given. }
		\FOR{$ t = 1, \ldots, k$}
		\STATE 
		\beq \label{deterministic_algorithm_step}
		x_{t+1} = \argmin_{x\in X} ~   \gamma_t \big\langle F(x_t) + \lambda_t \big( F(x_t) - F(x_{t-1}) \big) , x \big\rangle + V(x_t, x).
		\eeq
		\ENDFOR
	\end{algorithmic}
\end{algorithm} 
We add some remarks about the differences between the OE method 
with a few other existing VI methods, especially those with accelerated rate of convergence.
Firstly, while the classic extragradient method \cite{korpelevich1976extragradient} and mirror-prox method~\cite{Nem05-1} require at least two operator evaluations
and two projections, OE only requires one operator evaluation $F(x_t)$ and one projection (or prox-mapping) over the set $X$.
The more recent projected reflected gradient method for VI in~\cite{Malisky15} requires a simple recursion at each iteration given by 
\[
x_{t+1} = \argmin_{x\in X} ~   \gamma_t \big\langle F(x_t + \beta_t (x_t - x_{t-1})) + V(x_t, x).
\] 
This scheme implicitly maintains two
sequences, i.e., $\{x_t\}$ and $\{x_t + \beta_t(x_t - x_{t-1})\}$. Because the sequence  $\{x_t + \lambda_t(x_t - x_{t-1})\}$ may 
sit outside the feasible region $X$, it requires $F$ to be well-defined over the whole $\bbr^n$.
 Moreover,  all these methods in \cite{korpelevich1976extragradient,Malisky15, Nem05-1} output a solution (e.g., $x_t$) different from
the point where the operator is evaluated (e.g., $x_t + \beta_t(x_t - x_{t-1})$). As a consequence, it is difficult to directly utilize
the strong monotonicity conditions in \eqnok{G_monotone0} or \eqnok{G_monotone}.
Secondly, Nesterov and Scimali~\cite{NesterovScimali06} show that the dual extrapolation
method in \cite{Nest07-2} can be used to solve strongly monotone problems in an optimal way.
However, similar to \cite{Nem05-1}, each iteration of this method requires  two operator evaluations
and two projections. In addition, this method requires the strong monotonicity in \eqnok{strongly_monotone}
rather than the generalized strong monotonicity in \eqnok{G_monotone0} or~\eqnok{G_monotone}.

In order to analyze the convergence behavior of the OE method,
we first need to discuss a few different termination criteria
for the VI problem in \eqnok{VIP}.
If $F$ satisfies the generalized strong monotonicity condition in \eqnok{G_monotone} for some $\mu> 0$, then
the distance to the optimal solution $V(x_k, x^*)$
will be a natural termination criterion. \textcolor{black}{In addition we will use two other  termination criteria. The first termination criterion is called the weak gap, defined as 
\beq \label{def_weak_gap}
\gap(\bar{x}) := \max_{x \in X} \langle F(x),  \bar{x} -x \rangle
\eeq
 for a given $\bar x \in X$. We employ this criterion in the case of standard monotone VIs in Subsection~\ref{sec_MVI}. Our analysis in that subsection applies solely to the case when $X$ is a bounded set. 
The case of an unbounded feasible set is investigated in section of 5 of \cite{MonSva10-1}. In the context of generalized monotone VIs in Subsection ~ \ref{sec_GMVI} we consider  
also the residual of a point as a termination criterion and our analysis applies  also to possibly unbounded feasible sets. }
To this end
let us denote the normal cone of $X$ at $\bar x$ by
\beq \label{def_N_X}
N_X(\bar x) := \{y \in \bbr^n | \langle y, x - \bar x\rangle \le 0, \forall x \in X \}.
\eeq
Noting that $\bar x \in X$ is an optimal solution for problem~\eqnok{VIP} if and only if $F(\bar x) \in -N_X(\bar x)$,
we define the residual of $\bar x$ as
\beq \label{def_res}
\res(\bar x) := \min_{y \in - N_X(\bar x)} \|y - F(\bar x)\|_*.
\eeq
In particular, if $X = \bbr^n$, then $N_X(\bar x) = \{0\}$ and $\res(\bar x) = \|F(\bar x)\|_*$, which is exactly
 the residual of solving the nonlinear equation $F(\bar x) = 0$.

As stated, using these termination criteria, we will establish the convergence of
the OE method applied for solving different VI problems, including the generalized strongly monotone
VI (GSMVI), 
the generalized monotone VI (GMVI),
and 
the standard monotone VI (MVI) 
in Subsections~\ref{sec_GS_VI},~\ref{sec_GMVI} and~\ref{sec_MVI}, respectively.
We will first show in Proposition~\ref{lemma_dist} some important convergence properties associated with the OE method that hold for all these  cases.
Before doing so, we briefly state a well-known technical result (see, e.g., Lemma 3.1 of \cite{LanBook2020}), which, often referred to as the ``three-point lemma", characterizes the optimality condition of problem~\eqnok{deterministic_algorithm_step}.
\begin{lemma} \label{lemma_projection}
	Let $x_{t+1}$ be defined in \eqnok{deterministic_algorithm_step}. Then,
	\beq \label{opt_deter_step}
	\gamma_t \langle  F(x_t) + \lambda_t \big( F(x_t) - F(x_{t-1}) \big) , x_{t+1} - x \rangle + V(x_t, x_{t+1}) \leq V(x_t, x) - V(x_{t+1}, x), \forall x \in X.
	\eeq
\end{lemma}

Henceforth for a given sequence of iterates $ \{ x_t\} $ and $x \in X$ we will use the notation
\beq \label{def_Delta_V}
\Delta F_t := F(x_t) - F(x_{t-1}) \ \ \mbox{and} \ \ \Delta V_t(x)  :=  V(x_t, x) - V(x_{t+1}, x).
\eeq
\begin{proposition}\label{lemma_dist}
	Let $\{x_t\}$ be generated by Algorithm \ref{alg:deterministic} and $\{\theta_t\} $ a sequence of nonnegative numbers.
 	If the parameters $\{\gamma_t\}$ and $\{\lambda_t\}$ in Algorithm \ref{alg:deterministic} satisfy 
	\begin{align}
	\t_{t+1}\g_{t+1} \lambda_{t+1}= &~ \g_{t} \t_t, \label{eqn:OElambda}\\
	\theta_{t-1} \geq &~ 4 L^2 \theta_{t} \gamma_{t}^2 \lambda_{t}^2\label{eqn:OEGamma_strong}
	\end{align}
	for all $t = 1, \ldots, k$, then for any $x \in X$,
	\begin{align*}
	\tsum_{t=1}^k \big[ \theta_t \left[ \gamma_t \langle F(x_{t+1}) , x_{t+1} - x \rangle + V(x_{t+1},x) \right] \big] - L^2 \theta_k \gamma_k^2 \| x_{k+1} - x \|^2
	\leq \tsum_{t=1}^k \theta_t V(x_t,x).
	\end{align*}
\end{proposition}
\begin{proof} 
   It follows from \eqnok{opt_deter_step} after multiplying with $\theta_t$ \textcolor{black}{and invoking the definitions in \eqref{def_Delta_V}} that
		\begin{align}
	\nonumber
	\theta_t  \Delta V_t(x)  & \geq  \theta_t \gamma_t \langle F(x_{t+1}) , x_{t+1} - x \rangle  - 
	\theta_t \gamma_t \langle \Delta F_{t+1}, x_{t+1} - x \rangle + \theta_t \gamma_t \lambda_t  \langle\Delta F_t , x_{t} - x \rangle \\
	&\quad + \theta_t \gamma_t \lambda_t \langle\Delta F_t , x_{t+1} - x_t \rangle + \theta_t V(x_t, x_{t+1}).
	\label{eqn:1}
	\end{align}
	Summing up \eqref{eqn:1} from $t =  1$ to $k $, invoking \eqref{eqn:OElambda} and noting $x_1 = x_0$,  we obtain
		\begin{align}
		\tsum_{t=1}^k    \theta_t \Delta V_t(x)   \geq   \tsum_{t=1}^k \big[ \theta_t \gamma_t \langle F(x_{t+1}) , x_{t+1} - x \rangle \big] 
		- \theta_k \gamma_k \langle \Delta F_{k+1} , x_{k+1} - x \rangle   + Q_k, \label{eqn:common_bnd_VI}
	\end{align}	
	where 
	\beq \label{def_Q_t}
	 Q_k := \tsum_{t=1}^k  
	\big[ \theta_t \gamma_t \lambda_t \langle\Delta F_t, x_{t+1} - x_t \rangle + \theta_t V(x_t, x_{t+1}) \big].
	\eeq
	Using \eqnok{strong_convex_V}  and the Lipschitz condition \eqref{Lipschitz} we can lower bound the term $Q_k$ as follows:
		\begin{align*}
		Q_k&\geq     \tsum_{t=1}^k  \big[ -   \theta_t \gamma_t \lambda_t L \|   x_t - x_{t-1} \| 
		\|   x_{t+1} - x_t \|  +   \tfrac{\theta_t}{2} \|  x_t - x_{t+1} \|^2  \big] \\
		& \geq   \tsum_{t=1}^k  \big[-    \theta_t \gamma_t \lambda_t L \|   x_t - x_{t-1} \| 
		\|   x_{t+1} - x_t \|  +   \tfrac{\theta_t}{4} \|  x_t - x_{t+1} \|^2 +   \tfrac{\theta_{t-1} }{4} \|  x_t - x_{t-1} \|^2  \big]  +       \tfrac{ \theta_k}{4} \|  x_k - x_{k+1} \|^2  \\
		&\geq    \tfrac{\theta_k }{4} \|  x_k - x_{k+1} \|^2,
		\end{align*}
	where the last step follows by employing \eqref{eqn:OEGamma_strong}. Hence
	\begin{align*}
		\tsum_{t=1}^k    \theta_t \Delta V_t(x)  &\geq   \tsum_{t=1}^k \big[ \theta_t \gamma_t \langle F(x_{t+1}) , x_{t+1} - x \rangle \big]
		- \theta_k \gamma_k \langle \Delta F_{k+1} , x_{k+1} - x \rangle   +  \tfrac{\theta_k }{4} \|  x_k - x_{k+1} \|^2.
	\end{align*}
	Using the fact that 
	\begin{eqnarray*}
		- \theta_k \gamma_k \langle \Delta F_{k+1} , x_{k+1} - x \rangle  +
		\tfrac{\theta_k }{4} \|  x_k - x_{k+1} \|^2  
		&\geq& 
		- \theta_k \gamma_k L \|  x_k - x_{k+1} \|  \|  x - x_{k+1} \| +
		\tfrac{\theta_k }{4} \|  x_k - x_{k+1} \|^2  \\
		&\geq&  - L^2 \theta_k \gamma_k^2  \|  x - x_{k+1} \|^2	
	\end{eqnarray*}
	in the above inequality, we obtain the desired result.
\end{proof}

\subsection{Convergence for GSMVIs} \label{sec_GS_VI}
In this subsection, we consider the generalized strongly monotone VIs which satisfy \eqnok{Lipschitz} and \eqnok{G_monotone}
for some $\mu > 0$.

\begin{theorem}\label{the_linear_convergence}
Assume that \eqnok{G_monotone} holds for some $\mu > 0$.
	Let $x^*$ be a solution of problem  \eqref{VIP}, and suppose that
	 the parameters $\{\g_t\}$, $\{\t_t\}$, $\{\lambda_t\}$  satisfy \eqref{eqn:OElambda},  \eqref{eqn:OEGamma_strong} and additionally   
	\begin{align}
	\t_{t} \leq  &~ \t_{t-1} (1 + 2 \mu \gamma_{t-1}), \label{eqn:4}\\
	L^2 \g_k^2 \leq &~ \tfrac{1}{2}.\label{eqn:5}     
	\end{align}	
		Then for all $k \geq 1 $,
	\beq
	\label{linear_VI_result_nps}
	2 \mu \theta_k \gamma_k  V(x_{k+1},x^*) 
	\leq  \theta_1 V(x_1,x^*).
	\eeq
	In particular by setting for $t =1, \ldots, k$,
	\begin{equation}
	\label{deterministic_parameters}
	\gamma_t = \tfrac{1}{2 L},~~\lambda_t = \tfrac{\theta_{t-1} \gamma_{t-1}}{\theta_t \gamma_t } =  ( \tfrac{\mu}{L } +1 )^{-1} ,~~~\mbox{and} ~~~\theta_t = ( \tfrac{\mu}{L } +1 )^t,
	\end{equation}
	we have
	\beq \label{linear_VI_result}
	V(x_{k+1},x^*) 
	\leq  \tfrac{L}{\mu}  ( \tfrac{L }{L +  \mu})^{k-1} V(x_1,x^*).
	\eeq
\end{theorem}
\begin{proof}
	By combining Proposition~\ref{lemma_dist}, \eqnok{G_monotone} and \eqnok{eqn:4}, we obtain
	\begin{align*}
	\tsum_{t=1}^k \theta_t (2  \mu \gamma_t + 1) V(x_{t+1},x^*)   - L^2 \theta_k \gamma_k^2 \| x_{k+1} - x^* \|^2
	&\leq      \theta_1 V(x_1,x^*)  + \tsum_{t=2}^k \theta_t V(x_t,x^*) \\
	&\leq  \theta_1 V(x_1,x^*)  + \tsum_{t=2}^k 	  \theta_{t-1} (2 \mu \gamma_{t-1} + 1) V(x_t,x^*),
	\end{align*}
	which together with  \eqnok{strong_convex_V} and \eqnok{eqn:5}  lead to the desired inequality in \eqnok{linear_VI_result_nps}. 
	The choice of algorithmic parameters in \eqnok{deterministic_parameters} is compatible with the conditions \eqnok{eqn:OElambda}, \eqnok{eqn:OEGamma_strong},  \eqnok{eqn:4},
	and \eqnok{eqn:5} respectively, thus by substituting
	into the above relation we obtain the linear rate of convergence in \eqnok{linear_VI_result} for Algorithm \ref{alg:deterministic}.
\end{proof}

\vgap

In view of Theorem~\ref{the_linear_convergence},
the number of OE iterations required to find a solution $\bar x \in X$ s.t. $V(\bar x,x^*) \le \epsilon$ for GSMVIs is bounded
 by ${\cal O} \{L/\mu \log(1/\epsilon)\}$. This bound appears to be optimal and significantly outperforms the
 ${\cal O} \{(L/\mu)^2 \log(1/\epsilon)\}$ iteration complexity bound possessed by
 the projected operator (gradient) method in terms of their dependence on the condition
 number $L/\mu$ (see \cite{NesterovScimali06} for more discussions). 
 To the best of our knowledge, this is the first time this optimal complexity
 has been obtained for GSMVIs, while \cite{NesterovScimali06} established 
 a similar result 
 for strongly monotone VIs by using a more involved algorithmic scheme.
 
\subsection{Convergence for GMVI} \label{sec_GMVI}
In this subsection, we consider generalized  monotone VIs which satisfy \eqnok{Lipschitz} and \eqnok{G_monotone}
with $\mu = 0$. Our goal is to show the OE method is robust in the sense that it converges even if the
modulus $\mu$ is rather small. Throughout this subsection we assume that the distance generating function $\o$ is differentiable and
its gradient is $L_\o$ Lipschitz continuous, i.e.,
\beq \label{omega_Lip}
\|\nabla \o(x_1) - \nabla \o(x_2)\|_* \le L_\o \|x_1 - x_2\|, \ \ \forall x_1, x_2 \in X.
\eeq
We define the output solution $x_{R+1}$ of OE method as
\beq \label{choose_the_best}
\|x_{R+1} - x_R\|^2 + \|x_R - x_{R-1}\|^2 = \min_{t=1, \ldots, k} \left( \|x_{t+1} - x_t\|^2 + \|x_t - x_{t-1}\|^2\right).
\eeq
Lemma~\ref{lemma_bnd_res} provides a technical result regarding
the relation between the residual of $x_{R+1}$ and the summation of squared distances $\tsum_{t=1}^k \|x_{t+1}- x_t\|^2$.
\begin{lemma} \label{lemma_bnd_res}
Let $x_t$, $t=1, \ldots, k+1$, be generated by the OE method in Algorithm~\ref{alg:deterministic}.
Assume that $x_{R+1}$ is chosen according to \eqnok{choose_the_best}. If
\beq \label{cond_rel_two_iterate}
\tsum_{t=1}^k \|x_{t+1}-x_t\|^2 \le \delta,
\eeq
then 
\[
 \res(x_{R+1}) \le2 (L + \tfrac{L_\o}{\gamma_R} + L \lambda_R)\tfrac{\sqrt{2 \delta}}{\sqrt{k}},
\]
where $\res(\cdot)$ is defined in \eqnok{def_res}.
\end{lemma}
\begin{proof}
Observe that by the optimality condition of \eqnok{deterministic_algorithm_step}, we have
\beq \label{opt_det_step}
\langle F(x_{R+1}) + \delta_R, x - x_{R+1} \rangle \ge 0 \ \ \forall x \in X,
\eeq
with
\[
\delta_R := F(x_R) - F(x_{R+1}) + \lambda_R [F(x_R) - F(x_{R-1})] + \tfrac{1}{\gamma_R} [\nabla \o(x_{R+1}) - \nabla \o(x_{R})].
\]
By \eqnok{Lipschitz} and \eqnok{omega_Lip}, we have
\begin{align}
\|\delta_R\|_* &\le \|F(x_R) - F(x_{R+1}) \|_* + \lambda_R \|F(x_R) - F(x_{R-1})\|_* +   \tfrac{1}{\gamma_R} \|\nabla \o(x_{R+1}) - \nabla \o(x_{R})\|_* \nn \\
&\le (L + \tfrac{L_\o}{\gamma_R})\|x_{R+1} - x_R\| + L \lambda_R \|x_R - x_{R-1}\|. \label{bnd_delta_R}
\end{align}
It follows from \eqnok{cond_rel_two_iterate} that
\begin{align}
\tsum_{t=1}^k \left(\|x_{t+1} - x_t\|^2 + \|x_t - x_{t-1}\|^2 \right) \le 2 \tsum_{t=1}^2 \|x_{t+1} - x_t\|^2 \le 2 \delta.
\end{align}
The previous conclusion clearly implies that
$
\|x_{R+1} - x_R\|^2 + \|x_R - x_{R-1}\|^2 \le \tfrac{2 \delta}{k}
$
and hence that 
\beq \label{bnd_two_diff}
\max\{\|x_{R+1} - x_R\|, \|x_{R-1} - x_R\|\} \le \tfrac{\sqrt{2 \delta}}{\sqrt{k}}.
\eeq
We then conclude from the definition of $\res(\cdot)$ 
in \eqnok{def_res} and relations
 \eqnok{opt_det_step}, \eqnok{bnd_delta_R}, and \eqnok{bnd_two_diff} that
 \[
 \res(x_{R+1}) = \|\delta_R\|_* \le 2 (L + \tfrac{L_\o}{\gamma_R} + L \lambda_R)\tfrac{\sqrt{2 \delta}}{\sqrt{k}}.
 \]
\end{proof}

We are now ready to show the convergence of the OE method for GMVIs.
\begin{theorem}\label{lemma_dist_no_strong_monotone}
	Let $\{x_t\}$ be generated by Algorithm \ref{alg:deterministic} and $\{\theta_t\} $ be a sequence of nonnegative numbers.
	If the parameters $\{\gamma_t\}$ and $\{\lambda_t\}$ in Algorithm \ref{alg:deterministic} satisfy \eqnok{eqn:OElambda} and
	\beq \label{eqn:OEGamma_strong_no}
	\theta_{t-1} \geq  9 L^2 \theta_{t} \gamma_{t}^2 \lambda_{t}^2 
	\eeq
	for all $t = 1, \ldots, k$, then 
	\begin{align}
		 \tfrac{1 }{3} \tsum_{t=1}^k [\theta_t V(x_t, x_{t+1})] + \tfrac{\theta_k}{2} 
		(1-L^2\gamma_k^2) \| x_{k+1} - x^* \|^2 \le \theta_1 V(x_1,x^*). \label{GMVI_result}
	\end{align}
	In particular, if 
	\beq \label{stepsizeforGMV_no}
	\theta_t = 1, \lambda_t = 1 \ \ \mbox{and} \ \ \gamma_t = \tfrac{1}{3L},
	\eeq
	and $x_{R+1}$ is chosen according to \eqnok{choose_the_best},
	then 
\beq \label{bnd_res_no}
 \res(x_{R+1}) \le 4 L (2 + 3 L_\o )\tfrac{\sqrt{3 V(x_1,x^*)}}{\sqrt{k}}.
\eeq
where $\res(\cdot)$ is defined in \eqnok{def_res}.
\end{theorem}

\begin{proof} 
Observe that  \eqnok{eqn:common_bnd_VI}  still holds. However,
we will bound $Q_k$ in \eqnok{def_Q_t}  differently from the proof of Proposition~\ref{lemma_dist}.
By spliting $V(x_t, x_{t+1}) \ge \|x_{t+1} - x_t\|/2$ into three equal terms in $Q_k$
we obtain
\begin{eqnarray*}
		Q_k	&\geq&   -  \tsum_{t=1}^k    \left[  \theta_t \gamma_t \lambda_t L \|   x_t - x_{t-1} \| 
		\|   x_{t+1} - x_t \|  +  \theta_t V(x_t, x_{t+1})\right] \\
		 &=&  \tsum_{t=1}^k  \left[ -    \theta_t \gamma_t \lambda_t L \|   x_t - x_{t-1} \| 
\|   x_{t+1} - x_t \|  +   \tfrac{\theta_t}{6} \|  x_t - x_{t+1} \|^2 +   \tfrac{ \theta_{t-1}}{6} \|  x_t - x_{t-1} \|^2    \right]
+    \tfrac{\theta_k}{6} \|  x_k - x_{k+1} \|^2 \\ &&+  \tsum_{t=1}^k   \tfrac{\theta_t }{3} V(x_t, x_{t+1}) \\
&\ge& \tfrac{\theta_k}{6} \|  x_k - x_{k+1} \|^2 +  \tsum_{t=1}^k    \tfrac{\theta_t }{3} V(x_t, x_{t+1}).
	\end{eqnarray*} 
	Here, the first inequality follows from \eqref{Lipschitz} and \eqnok{strong_convex_V}, and
	the second inequality 
	follows
	from \eqnok{eqn:OEGamma_strong_no}. Pluging the above bound into \eqnok{eqn:common_bnd_VI} with $x = x^*$ and 
	using the fact $ \langle F(x_{t+1}) , x_{t+1} - x^* \rangle \ge 0$ due to \eqnok{G_monotone}, we have
	\begin{eqnarray*}
	\tsum_{t=1}^k    \theta_t \Delta V_t(x^*)  & \geq &  \tsum_{t=1}^k \big[ \theta_t \gamma_t \langle F(x_{t+1}) , x_{t+1} - x^* \rangle \big] 
		- \theta_k \gamma_k \langle \Delta F_{k+1} , x_{k+1} - x^* \rangle\\
		&&   + \tfrac{\theta_k}{6} \|  x_k - x_{k+1} \|^2 +  \tsum_{t=1}^k    \tfrac{\theta_t }{3} V(x_t, x_{t+1})\\
		&\ge& - \theta_k \gamma_k \langle \Delta F_{k+1} , x_{k+1} - x^* \rangle + \tfrac{\theta_k}{6} \|  x_k - x_{k+1} \|^2 +  \tsum_{t=1}^k    \tfrac{\theta_t }{3} V(x_t, x_{t+1}).
	\end{eqnarray*}
	Observing that
	\begin{eqnarray*}
		- \theta_k \gamma_k \langle \Delta F_{k+1} , x_{k+1} - x^* \rangle  +
		\tfrac{\theta_k }{6} \|  x_k - x_{k+1} \|^2  
		&\geq& 
		- \theta_k \gamma_k L \|  x_k - x_{k+1} \|  \|  x^* - x_{k+1} \| +
		\tfrac{\theta_k }{6} \|  x_k - x_{k+1} \|^2  \\
		&\geq&  - \tfrac{3 L^2}{2} \theta_k \gamma_k^2  \|  x^* - x_{k+1} \|^2,
	\end{eqnarray*}
	we then conclude from the previous inequality that
	\begin{eqnarray*}
	\tsum_{t=1}^k \big\{  \theta_t \big[ V(x_{t+1},x^*) + \tfrac{1 }{3} V(x_t, x_{t+1})\big]\big\} - \tfrac{3}{2} 
		L^2 \theta_k \gamma_k^2 \| x_{k+1} - x^* \|^2
	\leq \tsum_{t=1}^k \theta_t V(x_t,x^*),
	\end{eqnarray*}
	or equivalently,
	\[
	 \tfrac{1 }{3} \tsum_{t=1}^k [\theta_t V(x_t, x_{t+1})] + \theta_k V(x_{k+1},x^*)  - \tfrac{3}{2} 
		L^2 \theta_k \gamma_k^2 \| x_{k+1} - x^* \|^2 \le \theta_1 V(x_1,x^*).
	\]
	The result in \eqnok{GMVI_result} then follows from the above inequality and \eqnok{strong_convex_V}.
	Moreover, by choosing the parameter setting in \eqnok{stepsizeforGMV_no}, we have
	$
	\tsum_{t=1}^k   \tfrac{1}{3}V(x_t, x_{t+1}) \leq V(x_1,x^*),
 	$
	which implies that
	$
	\tsum_{t=1}^k   \| x_t -  x_{t+1}\|^2 \leq 6 V(x_1,x^*).
	$
	The result in \eqnok{bnd_res_no} immediately follows from the previous conclusion and Lemma~\ref{lemma_bnd_res}.
\end{proof}

\vgap

In view of Theorem~\ref{lemma_dist_no_strong_monotone}, the OE method
can find a solution $\bar x \in X$ s.t. $\res(\bar x) \le \epsilon$ in ${\cal O} (1/\epsilon^2)$
iterations for GMVIs. 

\subsection{Convergence for MVIs} \label{sec_MVI}
In this subsection, we show that  the OE method 
can achieve the optimal rate of convergence for solving standard monotone VIs,
for which $F$ satisifies
\beq \label{monotone}
\langle F(x_1) - F(x_2), x_1 - x_2 \rangle \ge 0, \ \ \forall x_1, x_2 \in X.
\eeq
\textcolor{black}{In this subsection it is assumed that the set $X$ is bounded.}
\begin{theorem} \label{the_monotone}
Suppose that \eqnok{monotone} holds. Let $\{x_t\}$ be generated by Algorithm \ref{alg:deterministic} and
denote
\beq \label{def_xbar}
\bar{x}_{k+1} := \tfrac{\tsum_{t=1}^k \gamma_t \theta_t x_{t+1}}{\tsum_{t=1}^k \gamma_t \theta_t }.
\eeq
If \eqref{eqn:OElambda}, \eqref{eqn:OEGamma_strong} and  \eqref{eqn:5} hold and $\theta_t \leq \theta_{t-1}$, then
$$
\gap(\bar{x}_{k+1}) 
\leq  \tfrac{\theta_1}{\tsum_{t=1}^k \theta_t  \gamma_t} \max_{x \in X} V(x_1,x).
$$
In particular, If $\gamma_t = \tfrac{1}{2 L}$, $\lambda_t = 1$ and $\theta_t =  1$,
then
$$
\gap(\bar{x}_{k+1}) 
\leq \tfrac{2L}{k} \max_{x \in X} V(x_1,x)
$$

\end{theorem}

\begin{proof}
In view of Proposition~\ref{lemma_dist}, 
$$
\tsum_{t=1}^k \theta_t \big[ \gamma_t \langle F(x_{t+1}) , x_{t+1} - x \rangle + V(x_{t+1},x) \big] - L^2 \theta_k \gamma_k^2 \| x_{k+1} - x \|^2
\leq \tsum_{t=1}^k \theta_t V(x_t,x).
$$
By using the monotonicity property in \eqnok{monotone} and 
the definition of $\bar x_{k+1}$ in \eqnok{def_xbar},
\begin{eqnarray*}
&& \tsum_{t=1}^k \theta_t \big[ \gamma_t \langle F(x_{t+1}) , x_{t+1} - x \rangle \big] \ge   \tsum_{t=1}^k \theta_t \big[ \gamma_t \langle F(x) , x_{t+1} - x \rangle \big]	
=(\tsum_{t=1}^k \theta_t  \gamma_t )  \langle F(x), \bar{x}_{k+1} - x  \rangle.
\end{eqnarray*}
Combining the above two inequalities, we obtain
$$
(\tsum_{t=1}^k \theta_t  \gamma_t )  \langle F(x), \bar{x}_{k+1} - x  \rangle + \tsum_{t=1}^{k-1} \theta_t V(x_{t+1},x) + \theta_k (\tfrac{1}{2} - L^2  \gamma_k^2) \| x_{k+1} - x \|^2   \leq \theta_1  V(x_1,x) +\tsum_{t=2}^k \theta_{t-1} V(x_t,x),
$$
which together with the assumption $\theta_t \le \theta_{t-1}$ and \eqref{eqn:5} imply that $(\tsum_{t=1}^k \theta_t  \gamma_t )  \langle F(x), \bar{x}_{k+1} - x  \rangle \le \theta_1  V(x_1,x)$.
The result follows from the above inequality and the definition of $\gap(\bar x_{k+1})$ in \eqnok{def_weak_gap}.
\end{proof}

In view of Theorem~\ref{the_monotone},
the OE method can achieve the same ${\cal O}(1/\epsilon)$ optimal complexity bound for MVI,
similarly to the mirror-prox method in \cite{Nem05-1} and the projected reflected gradient method in \cite{Malisky15}.
However, the algorithmic scheme of the OE method appears to be much simpler than the mirror-prox method.
In contrast to the projected reflected gradient method, the OE method does not require
$F$ to be well-defined over the whole $\bbr^n$ and hence is applicable to a broader class of MVI problems.

\section{Stochastic VIs} \label{sec_SOE}
In comparison to deterministic VIs, the algorithmic studies for solving stochastic VIs, especially
(generalized) strongly monotone VIs, are still quite limited. In this section we show that
the stochastic OE method can significantly improve
existing complexity for solving stochastic strongly monotone VIs, while
achieving the optimal complexity for solving stochastic monotone VIs.

\subsection{The stochastic OE method}

In this section, we consider the stochastic VI for which the operator $F$ can be accessed only
through a stochastic oracle. More specifically, given the current iterate $x_t$, the stochastic
oracle can generate a random vector $\tilde{F}(x_t, \xi_t)$
s.t. 
\[
 \mathbb{E}[\tilde{ F}(x_t, \xi_t)] = F(x_t) \ \ \mbox{and} \ \ \mathbb{E}[\| \tilde{F}(x_t, \xi_t) - F(x_t) \|_*^2] \leq  \sigma^2,
\]
where $\xi_t \in \bbr^d$ denotes a random \textcolor{black}{vector} independent of $x_1, \ldots, x_t$.

We also consider the mini-batch setting that is widely used in practice. In this setting, at 
search point $x_t$
we call the stochastic oracle $m_t$ times to generate the
i.i.d. samples $ \{\xi_{t,i}\}_{i=1, \hdots, m_t}$ and  corresponding values of $\{\tilde{F}(x_t, \xi_{t,i})\}_{i=1, \ldots, m_t}$,
and compute an unbiased estimator of $F(x_t)$ according to 
\beq
\label{mini_batch_operator}
\tilde{F}(x_t) = \tfrac{1}{m_t} \tsum_{t=1}^{m_t} \tilde{F}(x_t,\xi_{t,i}).
\eeq
Assuming that the norm $\|\cdot \|_*$ is Euclidean, we then conclude from 
the independence of the successive samples 
\beq \label{bound_varieance_minibatch}
 \mathbb{E}[\| \tilde{F}(x_t) - F(x_t) \|_*^2] \leq  \sigma_t^2:= \tfrac{\sigma^2}{m_t}.
\eeq
Obviously, the above mini-batch setting reduces to the standard single-oracle setting if $m_t = 1$ for all $t \ge 1$.
It should be noted, however, that the single-oracle setting does not require the Euclidean structure.

%
%
%
%

The stochastic operator extrapolation (SOE) method (see Algorithm~\ref{alg:stochastic}) is obtained by replacing the exact
operator $F$ with a stochastic estimator $\tilde F$ in the OE method.
\begin{algorithm}[H]  \caption{The Stochastic Operator Extrapolation (SOE) Method}  
	\label{alg:stochastic}
	\begin{algorithmic} 
		\STATE{Let $x_0 = x_1 \in X$, and the nonnegative parameters $\{ \gamma_t\}$ and $\{\lambda_t\}$ be given. }
		\FOR{$ t = 1, \ldots, k$}
		\STATE 
		\beq \label{stochastic_algorithm_step}
		x_{t+1} = \arg \min_{x \in X} \gamma_t \langle \tilde{F}(x_t) + \lambda_t [\tilde{F}(x_t) - \tilde{F}(x_{t-1})],x \rangle + V(x_t,x).
		\eeq
		\ENDFOR
	\end{algorithmic}
\end{algorithm} 
Throughout this section for given sequences $ \{ x_t\} $ and $\{ \xi_t\}$ we will use the notation
\beq \label{def_Delta_F_tilde}
\Delta \tilde F_t := \tilde{F}(x_t) - \tilde{F}(x_{t-1}) \ \ \mbox{and} \ \ \delta_t :=   \tilde{F}(x_t) - F(x_t),
\eeq
where $\delta_t$
denotes the error associated with the computation of the operator $F(x_t)$.
We still employ the notation $\Delta F_t$ and  $\Delta V_t(x)$ for $x \in X$, as introduced
in \eqnok{def_Delta_V}.
Similar to Lemma~\ref{opt_deter_step},
we characterize the optimality condition of \eqnok{stochastic_algorithm_step} as follows.
\begin{lemma} \label{lemma_projection_stoch}
Let $x_{t+1}$ be defined in \eqnok{stochastic_algorithm_step}, then \textcolor{black}{$\forall x \in X$}, 
\beq \label{opt_stoch_step}
\gamma_t \langle  \tilde F(x_t) + \lambda_t \big( \tilde F(x_t) - \tilde F(x_{t-1}) \big) , x_{t+1} - x \rangle + V(x_t, x_{t+1}) \leq V(x_t, x) - V(x_{t+1}, x).
\eeq
\end{lemma}
What follows is the counterpart of Proposition \ref{lemma_dist} for the stochastic case.
\begin{proposition} \label{prop_general_stohastic}
	Let $\{x_t\}$ be generated by the SOE method and $\{\theta_t\} $ be a sequence of nonnegative numbers.
 	If the parameters in this method satisfy 
 	\eqnok{eqn:OElambda} and 
	\beq \label{split_8}
	\theta_{t-1} \geq 16 L^2 \gamma_t^2 \lambda_t^2 \theta_t
	\eeq
	for all $t = 1, \ldots, k$, then for any $x \in X$,
	\begin{align}
	\tsum_{t=1}^k \big\{  \theta_t \big[ \gamma_t \langle \tilde F(x_{t+1}) , x_{t+1} - x \rangle + V(x_{t+1},x)+ \tfrac{1}{4}  V(x_t, x_{t+1})\big]\big\} - 2 L^2 \theta_k \gamma_k^2 \| x_{k+1} - x \|^2  \nn\\
	 - \theta_k \gamma_k \langle \delta_{k+1} - \delta_k , x_{k+1} - x \rangle\leq \tsum_{t=1}^k \theta_t V(x_t,x) + 2 \tsum_{t=1}^k \left( \theta_t \gamma_t^2 \lambda_t^2 \|\delta_t - \delta_{t-1} \|_*^2 \right). \label{prop_general_stochastic_result}
	\end{align}
\end{proposition}
\begin{proof}
It follows from \eqnok{opt_stoch_step} after multiplying with $ \theta_t$ that
\begin{align}
\theta_t \Delta V_t(x) & \ge \theta_t \gamma_t \langle \tilde F(x_{t+1}), x_{t+1} - x \rangle -
\theta_t \gamma_t \langle \Delta \tilde F_{t+1}, x_{t+1} - x \rangle + \theta_t \gamma_t \lambda_t \langle \Delta \tilde F_t, x_{t} - x \rangle \nn \\
& \quad + \theta_t \gamma_t \lambda_t \langle  \Delta \tilde F_t, x_{t+1} - x_t \rangle  + \theta_t V(x_{t+1},x_t). \label{3point_stochastic}
\end{align}
Summing up  from $t =  1$ to $k $, invoking \eqref{eqn:OElambda} and $x_1 = x_0$,  and assuming also $ \delta_1 = \delta_0$  we obtain
	\begin{align}
		\tsum_{t=1}^k    \theta_t \Delta V_t(x)   \geq   \tsum_{t=1}^k \big[ \theta_t \gamma_t \langle \tilde F(x_{t+1}) , x_{t+1} - x \rangle \big] 
		- \theta_k \gamma_k \langle \Delta \tilde F_{k+1} , x_{k+1} - x \rangle    +  \tilde Q_k, \label{eqn:common_bnd_VI_stoc}
	\end{align}	
	with
\begin{align}
	 \tilde Q_k &:= \tsum_{t=1}^k  
	\left[ \theta_t \gamma_t \lambda_t \langle\Delta \tilde F_t, x_{t+1} - x_t \rangle + \theta_t V(x_t, x_{t+1}) \right] \nn \\
	&= \tsum_{t=1}^k \left[ \theta_t \gamma_t \lambda_t \langle \left( \Delta F_t  +\delta_t - \delta_{t-1} \right), x_{t+1} - x_t \rangle + \theta_t V(x_t, x_{t+1}) \right]. \label{def_tildeQ_t}
	\end{align}
	Using  the Lipschitz condition \eqref{Lipschitz} and $x_1 = x_0$, we can lower bound the term $\tilde Q_k$ as follows
	\begin{align*}
		\tilde Q_k
		&\geq   \tsum_{t=1}^k    \left[ -   \theta_t \gamma_t \lambda_t L \|   x_t - x_{t-1} \| 
		\|   x_{t+1} - x_t \|  +   \theta_t V(x_t, x_{t+1}) +   \theta_t \gamma_t \lambda_t \langle\delta_t - \delta_{t-1}, x_{t+1} - x_t \rangle\right] \\
		 & \geq   \tsum_{t=1}^k  \left[-    \theta_t \gamma_t \lambda_t L \|   x_t - x_{t-1} \| 
		\|   x_{t+1} - x_t \|  +   \tfrac{\theta_t}{8} \|  x_t - x_{t+1} \|^2 +   \tfrac{\theta_{t-1} }{8} \|  x_t - x_{t-1} \|^2  \right]  +       \tfrac{ \theta_k}{8} \|  x_k - x_{k+1} \|^2  \\
		   & \quad +  \tsum_{t=1}^k \left[ \theta_t \gamma_t \lambda_t \langle\delta_t - \delta_{t-1}, x_{t+1} - x_t \rangle + \tfrac{\theta_{t} }{8} \|  x_t - x_{t+1} \|^2  \right]   + \tfrac{1}{4} \tsum_{t=1}^k\theta_t V(x_t, x_{t+1}) \\
		&\geq    \tfrac{ \theta_k}{8} \|  x_k - x_{k+1} \|^2 + \tsum_{t=1}^k \left[ \theta_t \gamma_t \lambda_t \langle\delta_t - \delta_{t-1}, x_{t+1} - x_t \rangle + \tfrac{\theta_{t} }{8} \|  x_t - x_{t+1} \|^2  \right]
		+ \tfrac{1}{4} \tsum_{t=1}^k\theta_t V(x_t, x_{t+1})\\
		&\geq \tfrac{ \theta_k}{8} \|  x_k - x_{k+1} \|^2 - 2 \tsum_{t=1}^k \left( \theta_t \gamma_t^2 \lambda_t^2 \|\delta_t - \delta_{t-1} \|_*^2 \right) + \tfrac{1}{4} \tsum_{t=1}^k\theta_t V(x_t, x_{t+1}),
	\end{align*} 
	where the second inequality follows from \eqnok{strong_convex_V}, the third inequality  from \eqnok{split_8} and the last one follows from  Young's inequality.
	Using the above bound of $\tilde Q_k$ in \eqnok{eqn:common_bnd_VI_stoc}, we obtain
	\begin{align}
\tsum_{t=1}^k    \theta_t \Delta V_t(x)  &\geq   \tsum_{t=1}^k \big[ \theta_t \gamma_t \langle \tilde F(x_{t+1}) , x_{t+1} - x \rangle \big]
- \theta_k \gamma_k \langle \Delta \tilde F_{k+1} , x_{k+1} - x \rangle +  \tfrac{ \theta_k}{8} \|  x_k - x_{k+1} \|^2  \nn \\
   & \quad -  2 \tsum_{t=1}^k \left( \theta_t \gamma_t^2 \lambda_t^2 \|\delta_t - \delta_{t-1} \|_*^2 \right) + \tfrac{1}{4} \tsum_{t=1}^k\theta_t V(x_t, x_{t+1}) \label{stochatc_bound_general_tmp1} \\
 &\ge \tsum_{t=1}^k \big[ \theta_t \gamma_t \langle \tilde F(x_{t+1}) , x_{t+1} - x \rangle \big] 
- 2 L^2 \theta_k \gamma_k^2  \|  x - x_{k+1} \|^2  -  \gamma_k \theta_k \langle \delta_{k+1} - \delta_k , x_{k+1} - x \rangle \nn \\
  & \quad -  2 \tsum_{t=1}^k \left( \theta_t \gamma_t^2 \lambda_t^2 \|\delta_t - \delta_{t-1} \|_*^2 \right) + \tfrac{1}{4} \tsum_{t=1}^k\theta_t V(x_t, x_{t+1}), \label{stochatc_bound_general_tmp}
\end{align}
	where the second inequality follows from
	\begin{align*}
		&- \gamma_k \langle \Delta \tilde F_{k+1} , x_{k+1} - x \rangle  +
		\tfrac{1 }{8} \|  x_k - x_{k+1} \|^2 
		=  - \gamma_k \langle \Delta  F_{k+1} , x_{k+1} - x \rangle    -  \gamma_k \langle \delta_{k+1} - \delta_k , x_{k+1} - x \rangle
		+ \tfrac{1 }{8} \|  x_k - x_{k+1} \|^2 \\
		&\geq 
		-  \gamma_k L \|  x_k - x_{k+1} \|  \|  x - x_{k+1} \| +
		\tfrac{1 }{8} \|  x_k - x_{k+1} \|^2  -  \gamma_k \langle \delta_{k+1} - \delta_k , x_{k+1} - x \rangle \\
		&\geq  - 2 L^2 \gamma_k^2  \|  x - x_{k+1} \|^2  -  \gamma_k \langle \delta_{k+1} - \delta_k , x_{k+1} - x \rangle \textcolor{black}{.}
	\end{align*}
	The result is obtained by rearranging the terms in \eqnok{stochatc_bound_general_tmp}
\end{proof}

\subsection{Convergence for stochastic GSMVIs}

We describe the main convergence properties of SOE for stochastic generalized strongly monotone VI,
i.e., when \eqnok{G_monotone} holds for some $\mu > 0$.

\begin{theorem} \label{the_stoch_GMVI}
Suppose that \eqnok{G_monotone} holds for some $\mu \ge 0$.
 	If the parameters in the SOE method satisfy 
 	\eqnok{eqn:OElambda}, \eqnok{split_8}, and
	\begin{align}
		\theta_{t} &\leq \theta_{t-1} (2 \mu \gamma_{t-1} + 1), ~~t = 1, \ldots, k, \label{gamma_strong_mon_stoch} \\
		8 L^2 \gamma_k^2 &\leq 1, \label{gamma_strong_mon_stoch_k}
	\end{align}
then 
\begin{align*}
 \theta_k (2 \mu \gamma_k + \tfrac{1}{2})  \bbe[V(x_{k+1},x^*)] + \tsum_{t=1}^{k-1}\tfrac{\theta_t}{4}  \bbe[V(x_t, x_{t+1})]  
   \leq \theta_1 V(x_1,x^*)  +  \textcolor{black}{  4   \tsum_{t=1}^k \left( \theta_t \gamma_t^2 \lambda_t^2 ( \sigma_{t-1}^2 +    \sigma_t^2) \right)}
	+ 2 \theta_k  \gamma_k^2 \sigma_k^2.
\end{align*}
\end{theorem}

\begin{proof}
Denote $\bar \xi_t = (\xi_{t,1}, \ldots, \xi_{t,m_t})$.
Let us fix $x = x^*$ and take expectation on both sides of \eqnok{prop_general_stochastic_result} w.r.t. $ \bar \xi_1, \ldots, \bar \xi_k$. Then,
\begin{align}
\tsum_{t=1}^k \big\{  \theta_t \big[ \gamma_t \bbe[ \langle \tilde F(x_{t+1}) , x_{t+1} - x^* \rangle] + \bbe[V(x_{t+1},x^*)]+ \tfrac{1}{4}  \bbe[V(x_t, x_{t+1})]\big]\big\} - 2 L^2 \theta_k \gamma_k^2\bbe[ \| x_{k+1} - x^* \|^2]  \nn\\
	- \theta_k \gamma_k \bbe [\langle \delta_{k+1} - \delta_k , x_{k+1} - x^* \rangle]
	\leq \tsum_{t=1}^k \theta_t \bbe[V(x_t,x^*)] + 2 \tsum_{t=1}^k \left( \theta_t \gamma_t^2 \lambda_t^2\bbe[ \|\delta_t - \delta_{t-1} \|_*^2] \right).\label{bnd_stoch_temp}
\end{align}
First note  that $x_t $ is a deterministic function of $\bar \xi_1, \hdots, \bar \xi_{t-1}$. By conditioning on $ \bar \xi_1, \hdots, \bar \xi_{t-1}$ and  using the law of iterated expectations it follows that 
$
\mathbb{E}[ \langle \delta_t , x_t - x^* \rangle ] = 0.
$ 
This implies that 
\begin{align*}
\bbe[ \langle \tilde F(x_{t+1}) , x_{t+1} - x^* \rangle] &= \bbe[\langle F(x_{t+1}) , x_{t+1} - x^* \rangle], \\
\bbe[\langle \delta_{k+1} - \delta_k , x_{k+1} - x^* \rangle] &= \bbe[\langle - \delta_k , x_{k+1} - x^* \rangle] 
=  \bbe[\langle - \delta_k, x_{k+1} - x_k \rangle].
\end{align*}
In addition, we have
\begin{align*}
\bbe[\|\delta_t - \delta_{t-1} \|_*^2] &\le 2 (\bbe[\|\delta_t\|_*^2 + \bbe[\|\delta_{t-1}\|_*^2) \le      \textcolor{black}{ 2  \sigma_{t-1}^2 +   2 \sigma_t^2}.
\end{align*}
Using these observations in \eqnok{bnd_stoch_temp}, we obtain
\begin{align*}
\tsum_{t=1}^k \big\{  \theta_t \big[ \gamma_t \bbe[ \langle F(x_{t+1}) , x_{t+1} - x^* \rangle] + \bbe[V(x_{t+1},x^*)]+ \tfrac{1}{4}  \bbe[V(x_t, x_{t+1})]\big]\big\} - 2 L^2 \theta_k \gamma_k^2\bbe[ \| x_{k+1} - x^* \|^2]  \nn\\
	+\theta_k \gamma_k \bbe [\langle \delta_k , x_{k+1} - x_k \rangle]\leq \tsum_{t=1}^k \theta_t \bbe[V(x_t,x^*)] +  \textcolor{black}{  4   \tsum_{t=1}^k \left( \theta_t \gamma_t^2 \lambda_t^2 ( \sigma_{t-1}^2 +    \sigma_t^2) \right)}
\end{align*}
which together with the fact 
\begin{align*}
\tfrac{1}{4}  \bbe[V(x_k, x_{k+1}) +\gamma_k \bbe [\langle \delta_t , x_{k+1} - x_k \rangle] &\ge \tfrac{1}{8} \bbe[ \|x_k - x_{k+1}\|^2] +\gamma_k \bbe [\langle \delta_k
 , x_{k+1} - x_k \rangle] \\
&\ge - 2 \gamma_k^2 \bbe[\|\delta_t\|_*^2] \ge - 2 \gamma_k^2 \sigma_k^2
\end{align*}
 imply after taking \eqnok{G_monotone} into account
\begin{align*}
\tsum_{t=1}^k \big\{ \theta_t (2 \mu \gamma_t  + 1) \bbe[V(x_{t+1},x^*)]\big\} + \tsum_{t=1}^{k-1}\tfrac{\theta_t}{4}  \bbe[V(x_t, x_{t+1})]  - 2 L^2 \theta_k \gamma_k^2\bbe[ \| x_{k+1} - x^* \|^2]  \nn\\
\leq \tsum_{t=1}^k \theta_t \bbe[V(x_t,x^*)] +  \textcolor{black}{  4   \tsum_{t=1}^k \left( \theta_t \gamma_t^2 \lambda_t^2 ( \sigma_{t-1}^2 +    \sigma_t^2) \right)}
	+ 2 \theta_k  \gamma_k^2 \sigma_k^2.
\end{align*}
Invoking \eqnok{gamma_strong_mon_stoch} and using \eqnok{strong_convex_V}, we have
\begin{align*}
 \theta_k (2 \mu \gamma_k + \tfrac{1}{2}) \bbe[V(x_{k+1},x^*)] + \tsum_{t=1}^{k-1}\tfrac{\theta_t}{4}  \bbe[V(x_t, x_{t+1})]  
  + (\tfrac{1}{4} - 2 L^2 \gamma_k^2) \theta_k \bbe[ \| x_{k+1} - x^* \|^2] \\
   \leq \theta_1 V(x_1,x^*)  +  \textcolor{black}{  4   \tsum_{t=1}^k \left( \theta_t \gamma_t^2 \lambda_t^2 ( \sigma_{t-1}^2 +    \sigma_t^2) \right)}
	+ 2 \theta_k  \gamma_k^2 \sigma_k^2,
\end{align*}
which, in view of \eqnok{gamma_strong_mon_stoch_k}, clearly implies the result.
\end{proof}

\vgap 

We now specify the selection of a few particular
stepsize policies for solving stochastic GSMVI problems.

\begin{corollary}\label{step_size_stoch_strong_mon_k_unknown}
Consider the single-oracle setting with $m_t = 1$. If
$$
t_0 = \tfrac{4 L}{  \mu}, ~~~ \gamma_t= \tfrac{1}{\mu ( t_0 + t -1 ) },~~~ \theta_t  = (t + t_0 + 1) ~  (t + t_0),~~~ \lambda_t = \tfrac{\theta_{t-1} \gamma_{t-1}}{\theta_t ~  \gamma_t},
$$
then
\begin{align*}
\bbe[V(x_{k+1},x^*)] 
   \leq \tfrac{2( t_0 + 1)  (t_0+2) V(x_1,x^*)}{ (k + t_0 + 1)  (k + t_0)}  +  \tfrac{8(4 k+1) \sigma^2}{\mu^2  (k + t_0 + 1)  (k + t_0)}.
\end{align*}
\end{corollary}

\begin{proof}
Note that \eqnok{eqn:OElambda} holds by the definition of $\lambda_t$.
Observe that
\begin{align*}
8 L^2 \gamma_t^2 &=  \tfrac{8 L^2}{\mu^2 ( t_0 + t -1 )^2 } \le \tfrac{8 L^2  }{\mu^2 t_0 ^2 } \le 1, \\
 16 L^2 \gamma_t^2 \lambda_t^2 \theta_t &= 16 L^2 \gamma_t^2 \tfrac{\theta_{t-1}^2 \gamma_{t-1}^2}{\theta_t^2 ~  \gamma_t^2} \theta_t
 = 16 L^2 \tfrac{\theta_{t-1}^2 \gamma_{t-1}^2}{\theta_t} \le 16 L^2 \gamma_{t-1}^2 \theta_{t-1} \le \theta_{t-1},
\end{align*}
and thus both  \eqnok{split_8} and \eqnok{gamma_strong_mon_stoch_k} hold.
In order to check \eqnok{gamma_strong_mon_stoch}, we observe that
\begin{align*}
\theta_{t-1} (2 \mu \gamma_{t-1} + 1) = \theta_{t-1} \left(  \tfrac{2} {t_0 + t -1 }+1 \right) =  \theta_{t-1} \tfrac{t_0 + t +1}{t_0 + t -1} =  \theta_t.
\end{align*}
The result then follows from Theorem~\ref{the_stoch_GMVI} and the following 
simple calculations.
\begin{align*}
\theta_k (2 \mu \gamma_k + \tfrac{1}{2}) &\ge \tfrac{\theta_k}{2} = \tfrac{1}{2} (k + t_0 + 1) ~  (k + t_0), \\
 \tsum_{t=1}^{k} \theta_t \gamma_t^2 \lambda_t^2 &= 	\tsum_{t=1}^{k} \tfrac{\theta_{t-1}^2 \gamma_{t-1}^2}{\theta_t} 
 = \tsum_{t=1}^{k} \tfrac{(t + t_0-1)^2}{(t + t_0-2)^2}   \tfrac{(t + t_0)}{(t + t_0 + 1)}
\tfrac{1}{\mu^2} \le \tfrac{2 k}{\mu^2},\\
\theta_k \gamma_k^2 &= \tfrac{(k + t_0 + 1) ~  (k + t_0)}{\mu^2 (  k+t_0 -1 )^2} \le \tfrac{2}{\mu^2}.
\end{align*}
\end{proof}

In view of Corollary \ref{step_size_stoch_strong_mon_k_unknown},
the number of iterations performed by the SOE method to find a solution $\bar x \in X$ s.t. $\mathbb{E}[V(\bar x,x^*)] \le \epsilon$ is bounded by 
$${\cal O}\{\max( \tfrac{L \sqrt{V(x_1,x^*)}}{\mu \sqrt{\epsilon}},  \tfrac{\sigma^2 }{\mu^2 \epsilon})\}.$$
In view of the above result one can expect the benefit of reducing the variance $\sigma$ per iteration by increasing the size of $m_t$
in terms of the resulting convergence rate. 

Note that when $\sigma = 0$ (i.e., the deterministic case), the convergence rate achieved in Corollary \ref{step_size_stoch_strong_mon_k_unknown} is not linear and hence
not optimal. Assuming that the total number of iterations $k$ is given in advance, we can select a novel stepsize policy that  improves this convergence rate.

\begin{corollary}
\label{step_size_stoch_strong_mon_k_known}
Consider the single-oracle setting with $m_t = 1$. 
If $k$ is fixed,
\[
\gamma_t =\gamma= \min\{\tfrac{1}{4L},\tfrac{q\log k}{\mu k}\},~ \theta_t  = (2\mu \gamma+1)^{t},~ \lambda_t=\tfrac{1}{2\mu\gamma+1}, ~ \mbox{with}~ q = 1+\tfrac{\log(  \tfrac{  \mu^2V(x_1,x^*) } {\sigma^2})}{\log k}
\]
then
\begin{align*}
\bbe[V(x_{k+1},x^*)] 
   \leq 2\big(1+\tfrac{\mu}{2L} \big)^{-k}V(x_1,x^*) + \tfrac{(2 + 8 q \log k) \sigma^2}{\mu^2 k} +\tfrac{4q^2 (\log k)^2 \sigma^2}{\mu^2k^2}.
\end{align*}
\end{corollary}

\begin{proof}
Note that \eqnok{eqn:OElambda} holds by the definition of $\lambda_t$.
Observe that
\[
8 L^2 \gamma_t^2  \leq  \tfrac{8L^2}{16 L^2} < 1, \ \ \mbox{and} \ \ \
 16 L^2 \gamma_t^2 \lambda_t^2 \theta_t = 
  16 L^2 \tfrac{\theta_{t-1}^2 \gamma_{t-1}^2}{\theta_t} \le 16 L^2 \gamma_{t-1}^2 \theta_{t-1} \le \theta_{t-1},
\]
and thus both  \eqnok{split_8} and \eqnok{gamma_strong_mon_stoch_k} hold.
Also, \eqnok{gamma_strong_mon_stoch} holds, due to
$
\theta_{t-1} (2 \mu \gamma_{t-1} + 1) =  \theta_t.
$
The result then follows from Theorem~\ref{the_stoch_GMVI} and the following 
 calculations.
\begin{align*}
\tfrac{\theta_k^{-1}\theta_1 }{2\mu \gamma + \tfrac{1}{2}}V(x_1,x^*)&\leq 2(2\mu\gamma+1)^{-k}V(x_1,x^*) \leq 2(1+\tfrac{\mu}{2L})^{-k}V(x_1,x^*) + \tfrac{2V(x_1,x^*)}{k^{q}}\\
&=2(1+\tfrac{\mu}{2L})^{-k}V(x_1,x^*) + \tfrac{2\sigma^2}{\mu^2k}\\
 \tfrac{\theta_k^{-1}\sigma^2}{2\mu\gamma+\tfrac{1}{2}}\tsum_{t=1}^{k} \theta_t \gamma_t^2 \lambda_t^2 &=\tfrac{\theta_k^{-1}\sigma^2}{2\mu\gamma+\tfrac{1}{2}} 	\tsum_{t=1}^{k} \tfrac{\theta_{t-1}^2 \gamma_{t-1}^2}{\theta_t} 
  \leq \tfrac{\gamma \sigma^2}{\mu} \leq \tfrac{\sigma^2q\log k}{\mu^2k}.
\end{align*}
\end{proof}

In view of corollary \ref{step_size_stoch_strong_mon_k_known},
the number of iterations performed by the SOE method to find a solution $\bar x \in X$ s.t. $\mathbb{E}[V(\bar x,x^*)] \le \epsilon$ is bounded by 
$${\cal O}\{ \max( \tfrac{L }{\mu} \log\tfrac{V(x_1,x^*)}{\epsilon},  \tfrac{\sigma^2 }{\mu^2 \epsilon} \log\tfrac{1}{\epsilon}) \}.$$
This complexity bound is nearly optimal, up to a logarithmic factor, for solving GSMVI problems. 
In order to obtain the optimal convergence for GSMVIs,
we need to develop a more advanced stepsize policy obtained by properly resetting the iteration
index to zero for the stepsize policy in Corollary~\ref{step_size_stoch_strong_mon_k_unknown}.
More specifically, the OE iterations will be grouped into epochs indexed by $s$, and each epoch contains
$k_s$ iterations.
A local iteration index $\tilde t$, which is set to $0$ whenever a new epoch starts,
will take place of $t$ in the definitions of $\gamma_t$ and $\theta_t$ in Corollary~\ref{step_size_stoch_strong_mon_k_unknown}. 

\begin{corollary}\label{restart_stepsize}
	Consider the single-oracle setting with $m_t = 1$. Set $t_0=\tfrac{4L}{\mu}$, and let
	$$k_s =\lc \max\{(2\sqrt{2}-1)t_0+4, ~\tfrac{2^{s+6}\sigma^2}{\mu^2 V(x_1,x^*)}\} \rc,~s\in \mathbb{Z}^+, ~~ K_0=0, ~~\mbox{and}~~K_s=\tsum_{s'=1}^{s}k_{s'}.
	$$
	For $t=1,2,...,$ introduce the epoch index $\tilde s$ and local iteration index $\tilde t$ such that
	$$\tilde s = \argmax_{\{s\in \mathbb{Z}^+\}}\mathbbm{1}_{\{K_{s-1}<t\le K_s\}}~~\mbox{and}~~
	\tilde t:=t-K_{\tilde s-1}.$$
	For the stepsize policy
	$$\gamma_t = \tfrac{1}{\mu(t_0+\tilde t-1)},~~~\theta_t = (\tilde t+t_0+1)(\tilde t+t_0) ,~~\mbox{and}~~\lambda_t = \begin{cases}
	& 0, ~~~~~~~~~~~\tilde t = 1, \\
	&\tfrac{\theta_{t-1}\gamma_{t-1}}{\theta_{t}\gamma_{t}},~~~   \tilde t\ge 2,
	\end{cases}$$
	it holds that
	$
	\bbe[V(x_{K_s+1},x^*)]\leq 2^{-s}V(x_1,x^*)
	$ for any $s \ge 1$.
\end{corollary}
\begin{proof}
	First we note  that in each epoch we use the stepsize policy of Corollary \ref{step_size_stoch_strong_mon_k_unknown}. This enables us to infer that, for $s=1,2,\ldots,$
	\begin{align} \label{convergence_epoch}
	\bbe[V(x_{K_{s}+1},x^*)] 
	\leq \tfrac{2( t_0 + 1)  (t_0+2) \bbe[V(x_{K_{s-1}+1},x^*)]}{ (k_s + t_0 + 1)  (k_s + t_0)}  +  \tfrac{8(4 k_s+1) \sigma^2}{\mu^2  (k_s + t_0 + 1)  (k_s + t_0)}.
	\end{align}
	As such by taking the specification  of $k_s$ into account, we obtain
	\begin{align*}
	\bbe[V(x_{K_{1}+1},x^*)] 
	\leq &~\tfrac{2( t_0 + 1)  (t_0+2) V(x_{1},x^*)}{ (k_1 + t_0 + 1)  (k_1 + t_0)}  +  \tfrac{8(4 k_1+1) \sigma^2}{\mu^2  (k_1 + t_0 + 1)  (k_1 + t_0)}\\
	\leq &~ \tfrac{2( t_0 + 1)  (t_0+2) V(x_{1},x^*)}{ (2\sqrt{2}t_0 + 5)  (2\sqrt{2}t_0+4)} +  \tfrac{32 \sigma^2}{\mu^2  (k_1 + t_0 + 1) }\\
	\leq & ~ \tfrac{V(x_1,x^*)}{4}+\tfrac{V(x_1,x^*)}{4}=\tfrac{V(x_1,x^*)}{2}.
	\end{align*}
	The desired convergence result follows  by recursively using 	
	\eqref{convergence_epoch}.
\end{proof}

In view of corollary \ref{restart_stepsize},
the number of epochs performed by the SOE method to find a solution $\bar x \in X$ s.t. $\mathbb{E}[V(\bar x,x^*)] \le \epsilon$ is bounded by $\log_2 (V(x_1, x^*)/\epsilon)$. 
Then together with the length of each epoch, the total number of iterations is bounded by
$${\cal O}\{ \max( \tfrac{L }{\mu} \log_2\tfrac{V(x_1,x^*)}{\epsilon},  \tfrac{\sigma^2 }{\mu^2 \epsilon} ) \},$$
which is optimal for solving GSMVI problems. Note also that
to simplify the implementation, we can replace the quantity $\sigma^2/(\mu^2 V(x_1,x^*))$
in the definition of $k_s$ by some estimation or simply by ${\cal O}(1)$.

\subsection{Convergence for stochastic GMVIs}

In this subsection, we consider stochastic generalized monotone VIs which satisfy \eqnok{Lipschitz}, \eqnok{G_monotone}
with $\mu = 0$ and \eqnok{omega_Lip}. Our goal is to show the SOE method is robust in the sense that it converges when the
modulus $\mu$ is really small. We use constant size mini-batch method whose operator is defined in
\eqnok{mini_batch_operator}. Assuming the norm $\|\cdot\|_*$ is Euclidean, we then conclude from \eqnok{bound_varieance_minibatch} 
that  $\bbe[\|\tilde F(x_t)-F(x_t)\|_*^2]\leq \tilde \sigma^2:=\tfrac{\sigma^2}{m}$, where $m$ is the constant mini-batch size.
 Throughout this subsection we assume that the distance generating function $\o$ is differentiable and
satisfies \eqnok{omega_Lip}.

We define the output solution of OE method as $x_{R+1}$, where $R$ is uniformly chosen from $\{2, 3, \ldots, k\}$.
Lemma~\ref{lemma_bnd_res_SOE} provides a technical result regarding
the relation between the residual of $x_{R+1}$ and the summation of squared distances $\tsum_{t=1}^k\bbe[ \|x_{t+1}- x_t\|^2]$. 

\begin{lemma} \label{lemma_bnd_res_SOE}
	Let $x_t$, $t=1, \ldots, k+1$, be generated by the SOE method in Algorithm~\ref{alg:stochastic} and fix the mini-batch size $m$.
	If
	\beq \label{cond_rel_two_iterate_SOE}
	\tsum_{t=1}^k \bbe[\|x_{t+1}-x_t\|^2] \le \tilde{\delta}
	\eeq 
and $R$ is uniformly chosen from $\{2, 3, \ldots, k\}$,	then 
\textcolor{black}{
	\[
	\bbe[\res(x_{R+1})^2] \le[4(1+\lambda_R)^2 + 4\lambda_R^2]\tilde \sigma^2 + 8\big[(L + \tfrac{L_\o}{\gamma_R})^2 + L^2\lambda_R^2\big]\tfrac{\tilde \delta}{k-1},
	\]}
	where $\res(\cdot)$ is defined \eqnok{def_res}.
\end{lemma}

\begin{proof}
	Observe that by the optimality condition of \eqnok{stochastic_algorithm_step}, we have
	\beq \label{opt_det_step_SOE}
	\langle F(x_{R+1}) + \tilde \delta_R, x - x_{R+1} \rangle \ge 0 ,\ \ \forall x \in X,
	\eeq
	with
	\begin{align*}
	\tilde \delta_R := &\tilde F(x_R) -F(x_{R+1}) + \lambda_R [\tilde F(x_R) - \tilde F(x_{R-1})] + \tfrac{1}{\gamma_R} [\nabla \o(x_{R+1}) - \nabla \o(x_{R})]\\
	= &(1+\lambda_R) [\tilde F(x_R)-F(x_R)]+ \lambda_R[\tilde F(x_{R-1}) -F(x_{R-1})]+[F(x_R) -F(x_{R+1})] \\
	&+ \lambda_R [F(x_R) - F(x_{R-1})] + \tfrac{1}{\gamma_R} [\nabla \o(x_{R+1}) - \nabla \o(x_{R})].
	\end{align*}
	By \eqnok{Lipschitz} and \eqnok{omega_Lip}, we have
	\begin{align*}
	\|\tilde \delta_R\|_* 
	\le&~  (1+\lambda_R) \|\tilde F(x_R)-F(x_R)\| +  \lambda_R \|\tilde F(x_{R-1}) -F(x_{R-1})\|+ (L + \tfrac{L_\o}{\gamma_R})\|x_{R+1} - x_R\| + L \lambda_R \|x_R - x_{R-1}\|.
	\end{align*} 
	\textcolor{black}{By using Young's inequality, we obtain
	\begin{align*}
		\|\tilde \delta_R\|_*^2 
		\le&~  4(1+\lambda_R)^2 \|\tilde F(x_R)-F(x_R)\|^2 +  4\lambda_R^2 \|\tilde F(x_{R-1}) -F(x_{R-1})\|^2\\&~+ 4(L + \tfrac{L_\o}{\gamma_R})^2\|x_{R+1} - x_R\|^2 + 4L^2 \lambda_R^2 \|x_R - x_{R-1}\|^2.
	\end{align*} 
	Taking expectation on both sides of the above inquality, 
	then we have
	\begin{align}
	\bbe[\|\tilde \delta_R\|_*^2] \le&~[4(1+\lambda_R)^2 + 4\lambda_R^2]\tilde \sigma^2 + 4 \big[(L + \tfrac{L_\o}{\gamma_R})^2 + L^2\lambda_R^2\big]\big(\bbe[\|x_{R+1} - x_R\|^2]+\bbe[\|x_R - x_{R-1}\|^2]\big). \label{bnd_delta_R_SOE_2} 
	\end{align}
Let $ \mathcal{F}_t = \sigma(x_1, \hdots, x_t)$ denote the sigma field generated by the first $ x_1, \hdots, x_t $ iterates. Since $R$ is uniformly chosen from $\{2, 3,\ldots,k\}$ independently of $ \mathcal{F}_k$, we have
$$
\bbe[\|x_{R+1} - x_R\|^2+\|x_{R-1} - x_R\|^2] = \bbe\big[ \bbe[\|x_{R+1} - x_R\|^2+\|x_{R-1} - x_R\|^2] | \mathcal{F}_k ] \big] \leq   \tfrac{2}{k-1} \tsum_{t=1}^k\bbe[\|x_{t+1}-x_t\|^2].
$$} 
	As such
	\beq \label{bnd_sum_two_diff_SOE}
	\bbe[\|x_{R+1} - x_R\|^2+\|x_{R-1} - x_R\|^2]   \le  \tfrac{2 \tilde \delta}{k-1}.
	\eeq
	We then conclude from the definition of $\res(\cdot)$ 
	in \eqnok{def_res} and relations
	\eqnok{opt_det_step_SOE}, \eqnok{bnd_delta_R_SOE_2}, and \eqnok{bnd_sum_two_diff_SOE} that
	\[\bbe[\res(x_{R+1})^2] =\bbe[\|\tilde \delta_R\|_*^2] \le[4(1+\lambda_R)^2 + 4\lambda_R^2]\tilde \sigma^2 + 8\big[(L + \tfrac{L_\o}{\gamma_R})^2 + L^2\lambda_R^2\big]\tfrac{\tilde \delta}{k-1}.
	\]
\end{proof}

We can now show the convergence of the SOE method by establishing the convergence of $\bbe[\res(x_{R+1})] $.
\begin{theorem}\label{lemma_dist_no_strong_monotone_SOE}
	Let $\{x_t\}$ be generated by Algorithm \ref{alg:stochastic} and $\{\theta_t\} $ be a sequence of nonnegative numbers.
	If the parameters $\{\theta_t\}$, $\{\gamma_t\}$ and $\{\lambda_t\}$ in Algorithm \ref{alg:stochastic} satisfy \eqnok{eqn:OElambda}, \eqnok{split_8}, \eqnok{gamma_strong_mon_stoch_k} and $\theta_{t}\geq \theta_{t-1}$,
	for all $t = 1, \ldots, k$, and fix the mini-batch size $m$, then 
	\begin{align}
	\tsum_{i=1}^k \tfrac{\theta_t}{8}\bbe[\|x_{t+1}-x_t\|^2 ]\le \theta_1 V(x_1,x^*) + \tsum_{i=1}^k 8\theta_t \gamma_t^2\lambda_t^2\tilde \sigma^2 + 4\theta_k\gamma_k^2\tilde \sigma^2. \label{GMVI_result_SOE}
	\end{align}
	In particular, if 
	\beq \label{stepsizeforGMV_no_SOE}
	m=k+1, \\\theta_t = 1,\\ \lambda_t = 1 \ \ \mbox{and} \ \ \gamma_t = \tfrac{1}{4L},
	\eeq
	and $R$ is uniformly chosen from $\{2, 3,\ldots,k\}$,
	then 
\textcolor{black}{
	\beq \label{bnd_res_no_SOE}
	\bbe[\res(x_{R+1})^2] \le \tfrac{20\sigma^2}{k+1}+ 32\big[ (L+4LL_\o)^2 + L^2 \big]\tfrac{2V(x_1,x^*)+\frac{\sigma^2}{L^2}}{k-1},
	\eeq}
	where $\res(\cdot)$ is defined \eqnok{def_res}.
\end{theorem}

\begin{proof} 
	Observe that  \eqnok{stochatc_bound_general_tmp1} still holds. Now
	we bound the terms $- \theta_k\gamma_k \langle \Delta \tilde F_{k+1} , x_{k+1} - x \rangle  +
	\tfrac{\theta_k }{8} \|  x_k - x_{k+1} \|^2 $ on the right hand side of \eqnok{stochatc_bound_general_tmp1} in a slightly different manner than before, 
	\begin{align*}
	&- \gamma_k \langle \Delta \tilde F_{k+1} , x_{k+1} - x \rangle  +
	\tfrac{1 }{8} \|  x_k - x_{k+1} \|^2 
	=  - \gamma_k \langle \Delta  F_{k+1} , x_{k+1} - x \rangle    -  \gamma_k \langle \delta_{k+1} - \delta_k , x_{k+1} - x \rangle
	+ \tfrac{1 }{8} \|  x_k - x_{k+1} \|^2 \\
	&\geq 
	-  \gamma_k L \|  x_k - x_{k+1} \|  \|  x - x_{k+1} \| +
	\tfrac{1 }{16} \|  x_k - x_{k+1} \|^2+
	\tfrac{1 }{16} \|  x_k - x_{k+1} \|^2  -  \gamma_k \langle \delta_{k+1} - \delta_k , x_{k+1} - x \rangle \\
	&\geq  - 4 L^2 \gamma_k^2  \|  x - x_{k+1} \|^2 +
	\tfrac{1 }{16} \|  x_k - x_{k+1} \|^2 -  \gamma_k \langle \delta_{k+1} - \delta_k , x_{k+1} - x \rangle.
	\end{align*}
	Plugging the bound into \eqnok{stochatc_bound_general_tmp1} and rearranging the terms, we observe that
	\begin{align*}
	\tsum_{t=1}^k \big\{  \theta_t \big[ \gamma_t \langle \tilde F(x_{t+1}) , x_{t+1} - x \rangle + V(x_{t+1},x)+ \tfrac{1}{4}  V(x_t, x_{t+1})\big]\big\} - 4 L^2 \theta_k \gamma_k^2 \| x_{k+1} - x \|^2 + \tfrac{\theta_k}{16}\|x_{k+1}-x_k\|^2 \nn\\
	- \theta_k \gamma_k \langle \delta_{k+1} - \delta_k , x_{k+1} - x \rangle\leq \tsum_{t=1}^k \theta_t V(x_t,x) + 2 \tsum_{t=1}^k \left( \theta_t \gamma_t^2 \lambda_t^2 \|\delta_t - \delta_{t-1} \|_*^2 \right). 
	\end{align*}
	Fixing $x=x^*$, taking expectation on both sides of the inequality and using the fact $ \langle F(x_{t+1}) , x_{t+1} - x^* \rangle \ge 0$, we have
	\begin{eqnarray*}
		\tsum_{t=1}^k \big\{  \theta_t \big[ \bbe[V(x_{t+1},x^*)] + \tfrac{1 }{4} \bbe[V(x_t, x_{t+1}])\big]\big\} - 4
		L^2 \theta_k \gamma_k^2 \bbe[\| x_{k+1} - x^* \|^2]
		\\\leq \tsum_{t=1}^k \theta_t \bbe[V(x_t,x^*)] +  \tsum_{t=1}^k 8\theta_t\gamma_t^2\lambda_t^2\tilde \sigma^2
		+ \theta_k \gamma_k \bbe[\langle \delta_{k+1} - \delta_k , x_{k+1} - x^* \rangle] - \tfrac{\theta_k}{16}\bbe[\|x_{k+1}-x_k\|^2].
	\end{eqnarray*}
	By taking into account that
	\begin{align*}
	&\gamma_k \bbe[\langle \delta_{k+1} - \delta_k , x_{k+1} - x^* \rangle] - \tfrac{1}{16}\bbe[\|x_{k+1}-x_k\|^2]\\
	=~&\gamma_k \bbe[\langle  - \delta_k , x_{k+1} - x_k \rangle] - \tfrac{1}{16}\bbe[\|x_{k+1}-x_k\|^2]
	\le 4 \gamma_k^2\|\delta_k\|_*^2\le 4\gamma_k^2\tilde \sigma^2,
	\end{align*}
	we obtain
	\begin{eqnarray*}
		\tsum_{t=1}^k \big\{  \theta_t \big[ \bbe[V(x_{t+1},x^*)] + \tfrac{1 }{4} \bbe[V(x_t, x_{t+1}])\big]\big\} - 4
		L^2 \theta_k \gamma_k^2 \bbe[\| x_{k+1} - x^* \|^2]
		\\\leq \tsum_{t=1}^k \theta_t \bbe[V(x_t,x^*)] +  \tsum_{t=1}^k 8\theta_t\gamma_t^2\lambda_t^2\tilde \sigma^2+ 4\theta_k \gamma_k^2\tilde \sigma^2.
	\end{eqnarray*}
	Moreover, by using the condition \eqnok{gamma_strong_mon_stoch_k}, \eqnok{strong_convex_V}, and $\theta_{t}\geq \theta_{t-1}$, we obtain
	$$
	\tsum_{i=1}^k \tfrac{\theta_t}{8}\bbe[\|x_{t+1}-x_t\|^2 ]\le \theta_1 V(x_1,x^*) + \tsum_{i=1}^k 8\theta_t \gamma_t^2\lambda_t^2\tilde \sigma^2 + 4\theta_k\gamma_k^2\tilde \sigma^2.$$
	Finally, the choice of parameters in \eqnok{stepsizeforGMV_no_SOE} implies that
	$$
	\tsum_{i=1}^k \bbe[\|x_{t+1}-x_t\|^2 ]\le 8V(x_1,x^*) + \tsum_{i=1}^k \tfrac{4}{L^2} \tilde\sigma^2 + \tfrac{2}{L^2}\tilde\sigma^2\le 8V(x_1,x^*)  + \tfrac{4}{L^2}\sigma^2.
	$$
	The results in  \eqnok{bnd_res_no_SOE} readily follow from the previous conclusion and Lemma~\ref{lemma_bnd_res_SOE}.
\end{proof}

\vgap

In view of Theorem~\ref{lemma_dist_no_strong_monotone_SOE}, the SOE method
can find a solution $\bar x \in X$ such that $\bbe[\res(\bar x)^2] \le \epsilon$ in ${\cal O} (1/\epsilon)$
iterations and ${\cal O} (1/\epsilon^2)$ overall sample complexity for solving generalized stochastic monotone VIs. \textcolor{black}{ Our result  outperforms  \cite{iusem2017extragradient}, improving the rate  by a  logarithmic factor}.

\subsection{Convergence for stochastic MVIs}
\textcolor{black}{In this subsection it is assumed that the set $X$ is bounded.}
We start out with an auxiliary result that is helpful to the subsequent developments. 

\begin{lemma} \label{auxiliary_sequence}
	For $t = 1,2, \hdots, k, $ define $\hat{x}_{t+1}$ as 
	$$
	\hat{x}_{t+1} = \arg \min_{x \in X} \langle - \hat{\gamma}_t \delta_t, x \rangle + 
	V(\hat{x}_t,x),
	$$
	where $ \hat{x}_1 \in X$ is an arbitrary point and $ \{\hat{\gamma}_t \}$ a sequence of non-negative numbers. Then
	\begin{equation}
	\label{eqn:9}
	- \tsum_{t=1}^k  \langle  \hat{\gamma}_t \delta_t , \hat{x}_t - x \rangle  \leq
	V(\hat{x}_1,x) + \tsum_{t=1}^k  \tfrac{\hat{\gamma}_t^2}{2} \| \delta_t \|^2_*.
	\end{equation}
\end{lemma}

\begin{proof}
	The optimality condition is characterized as 
	$$
	\langle - \hat{\gamma}_t \delta_t , \hat{x}_{t+1} - x \rangle + V(\hat{x}_t, \hat{x}_{t+1})
	+V(\hat{x}_{t+1},x) \leq V(\hat{x}_{t},x).
	$$
	As such 
	$$
	\langle - \hat{\gamma}_t \delta_t , \hat{x}_t - x \rangle 
	- \hat{\gamma}_t \| \delta_t \|_*  \| \hat{x}_t - \hat{x}_{t+1} \| + \tfrac{1}{2}
	\| \hat{x}_t - \hat{x}_{t+1} \|^2 +V(\hat{x}_{t+1},x) \leq V(\hat{x}_{t},x) 
	$$
	and therefore 
	$$
	\langle - \hat{\gamma}_t \delta_t , \hat{x}_t - x \rangle 
	+V(\hat{x}_{t+1},x) \leq V(\hat{x}_{t},x) + \tfrac{\hat{\gamma}_t^2}{2} \| \delta_t \|^2_*.
	$$
	Summing up from $t = 1$ to $k$ gives the desired result.
\end{proof}

\begin{proposition} \label{prop_monotone_stohastic}
	Assume $ \max_{x_1,x_2 \in X} V(x_1,x_2) \leq D_X$.
	Let $\{x_t\}$ be generated by the SOE method and set $ \theta_t  = 1$.
	If the parameters in this method satisfy 
	\eqnok{eqn:OElambda}, and 
	$$
	1 \geq 16 L^2 \gamma_t^2 \lambda_t^2 ,  ~~~\gamma_t  = \gamma_{t+1} \lambda_{t+1}
	$$
	for all $t = 5, \ldots, k$, then 
	\beq
	\label{expected_gap_general}
	\mathbb{E}[\gap(\bar{x}_{k+1})]  \leq \frac{1}{(\tsum_{t= \lc \frac{k}{2} \rc}^k \gamma_t)} \big(
	\tfrac{9}{2} ~ \tsum_{t= \lc \frac{k}{2} \rc}^k   \gamma_{t-1}^2  \sigma^2  + \textcolor{black}{\tfrac{15}{4}  } D_X + 
	\tfrac{5}{2}  \gamma_k^2 \sigma^2 +	\gamma_k \gamma_{k-1} \sigma^2 \big).
	\eeq
	In particular by setting  	
	$
	\gamma_t = \tfrac{1}{L \sqrt{t}},
	$
	it follows that 
	\beq
	\label{expected_gap_fixed}
	\mathbb{E}[\gap(\bar{x}_{k+1})]  \leq \tfrac{2L }{\sqrt{k+1}} \bigg(	\tfrac{9}{2} \log{5}~ \tfrac{\sigma^2}{L^2}  + \textcolor{black}{\tfrac{15}{4}} D_X 
	+	\tfrac{7}{k} \tfrac{\sigma^2}{L^2}\bigg).
	\eeq
\end{proposition}
\begin{proof}
We start  from \eqnok{opt_stoch_step} and sum from  $t =  \bar{k} =  \lc \frac{k}{2} \rc$ to $k$, while invoking \eqref{eqn:OElambda}.
\begin{align}
	\tsum_{t= \bar{k}}^k    \Delta V_t(x)   \geq   \tsum_{t=\bar{k}}^k \big[  \gamma_t \langle \tilde F(x_{t+1}) , x_{t+1} - x \rangle \big] 
	-  \gamma_k \langle \Delta \tilde F_{k+1} , x_{k+1} - x \rangle    +         \gamma_{\bar{k}}  \lambda_{\bar{k}}    \langle  \Delta \tilde F_{\bar{k}}, x_{\bar{k}} - x \rangle             +   \tilde Q_k,
	\label{3point_stoch_new}
\end{align}	
with
\begin{align*}
	\tilde Q_k &:= \tsum_{t=\bar{k}}^k  
	\left[ \gamma_t \lambda_t \langle\Delta \tilde F_t, x_{t+1} - x_t \rangle +  V(x_t, x_{t+1}) \right].
\end{align*}	
Using  the Lipschitz condition \eqref{Lipschitz}, we can lower bound the term $\tilde Q_k$ as follows
\begin{align*}
	\tilde Q_k
	&\geq   \tsum_{t=\bar{k}}^k   \left[ -    \gamma_t \lambda_t L \|   x_t - x_{t-1} \| 
	\|   x_{t+1} - x_t \|  +   V(x_t, x_{t+1}) +    \gamma_t \lambda_t \langle\delta_t - \delta_{t-1}, x_{t+1} - x_t \rangle\right ] \\
	& \geq   \tsum_{t=\bar{k}}^k \left[-     \gamma_t \lambda_t L \|   x_t - x_{t-1} \| 
	\|   x_{t+1} - x_t \|  +   \tfrac{1}{8} \|  x_t - x_{t+1} \|^2 +   \tfrac{1 }{8} \|  x_t - x_{t-1} \|^2  \right]  +       \tfrac{ 1}{8} \|  x_k - x_{k+1} \|^2   -     \tfrac{ 1}{8} \|  x_{\bar{k}} - x_{{\bar{k}}-1} \|^2   \\
	& \quad +  \tsum_{t=\bar{k}}^k   \left[  \gamma_t \lambda_t \langle\delta_t - \delta_{t-1}, x_{t+1} - x_t \rangle + \tfrac{1 }{8} \|  x_t - x_{t+1} \|^2  \right]   + \tfrac{1}{4} 
	\tsum_{t=\bar{k}}^k  V(x_t, x_{t+1})    \\
	&\geq \tfrac{ 1}{4} \|  x_k - x_{k+1} \|^2  -      \tfrac{ 1}{8} \|  x_{\bar{k}} - x_{{\bar{k}}-1} \|^2  -   
	2 \tsum_{t=\bar{k}}^k \left(  \gamma_t^2 \lambda_t^2 \|\delta_t - \delta_{t-1} \|_*^2 \right) + \tfrac{1}{4} \tsum_{t=\bar{k}}^{k-1}  V(x_t, x_{t+1}),
\end{align*} 
where the individuals steps are similar to the ones encountered in the proof of Proposition \ref{prop_general_stohastic}.
Using the above bound of $\tilde Q_k$ in \eqref{3point_stoch_new}, we obtain
\begin{align}
	\tsum_{t=\bar{k}}^k     \Delta V_t(x)  &\geq \tsum_{t=\bar{k}}^k   \big[  \gamma_t \langle \tilde F(x_{t+1}) , x_{t+1} - x \rangle \big]
	-  \gamma_k \langle \Delta \tilde F_{k+1} , x_{k+1} - x \rangle    +         \gamma_{\bar{k}}  \lambda_{\bar{k}}    \langle  \Delta \tilde F_{\bar{k}}, x_{\bar{k}} - x \rangle  \nn   \\
&	\quad  +  \tfrac{ 1}{4} \|  x_k - x_{k+1} \|^2       -     \tfrac{ 1}{8} \|  x_{\bar{k}} - x_{{\bar{k}}-1} \|^2  -  2 \tsum_{t=\bar{k}}^k \left(  \gamma_t^2 \lambda_t^2 \|\delta_t - \delta_{t-1} \|_*^2 \right) + \tfrac{1}{4} \tsum_{t=\bar{k}}^{k-1}  V(x_t, x_{t+1}).  
\label{intermediate}
\end{align}
Note that 
\begin{eqnarray*}
 - \gamma_k \langle \Delta \tilde F_{k+1} , x_{k+1} - x \rangle  + \tfrac{1 }{4} \|  x_k - x_{k+1} \|^2   &=& - \gamma_k \langle \Delta  F_{k+1} , x_{k+1} - x \rangle    -  \gamma_k \langle \delta_{k+1} - \delta_k , x_{k+1} - x \rangle + \tfrac{1 }{4} \|  x_k - x_{k+1} \|^2  \nn  \\
 & \geq &  	-  \gamma_k L \|  x_k - x_{k+1} \|  \|  x - x_{k+1} \| +  \tfrac{1 }{8} \|  x_k - x_{k+1} \|^2 \nn  \\ &&  - \gamma_k \langle \delta_{k+1}  , x_{k+1} - x \rangle + \gamma_k \langle  \delta_k , x_{k} - x \rangle +
\gamma_k \langle  \delta_k , x_{k+1} - x_k  \rangle  + \tfrac{1 }{8} \|  x_k - x_{k+1} \|^2     \nn \\
& \geq &  - 2 L^2 \gamma_k^2  \|  x - x_{k+1} \|^2    - \gamma_k \langle \delta_{k+1}  , x_{k+1} - x \rangle + \gamma_k \langle  \delta_k , x_{k} - x \rangle  - 2 \gamma_k^2 \| \delta_k \|_*^2.  
\label{first_bound}
\end{eqnarray*}
Similarly 
\begin{eqnarray*}
 \gamma_{\bar{k}}  \lambda_{\bar{k}}    \langle  \Delta \tilde F_{\bar{k}}, x_{\bar{k}} - x \rangle   &=&    \gamma_{\bar{k}}  \lambda_{\bar{k}}    \langle  \Delta  F_{\bar{k}}, x_{\bar{k}} - x \rangle 
 +  \gamma_{\bar{k}}  \lambda_{\bar{k}}    \langle \delta_{\bar{k}} - \delta_{\bar{k}-1} ,   x_{\bar{k}} - x \rangle  \nn  \\
 & \geq & -    \gamma_{\bar{k}}  \lambda_{\bar{k}}    \| x_{\bar{k}} -  x_{\bar{k}-1}  \|    \| x_{\bar{k}} - x \|  +    \gamma_{\bar{k}}  \lambda_{\bar{k}}    \langle \delta_{\bar{k}} - \delta_{\bar{k}-1} ,   x_{\bar{k}}  -x \rangle
 +  \tfrac{1}{4}  \|   x_{\bar{k}} - x_{\bar{k}-1} \|^2 -   \tfrac{1}{4}  \|   x_{\bar{k}} - x_{\bar{k}-1} \|^2 \nn \\
 & \geq &  - 2 L^2  \gamma_{\bar{k}}^2  \lambda_{\bar{k}}^2   \| x_{\bar{k}} - x \|^2  +   \gamma_{\bar{k}}  \lambda_{\bar{k}}    \langle \delta_{\bar{k}}  ,   x_{\bar{k}}  -x \rangle   
 - 2 \gamma_{\bar{k}}^2   \lambda_{\bar{k}}^2  \| \delta_{\bar{k}-1} \|_*^2 - \gamma_{\bar{k}}  \lambda_{\bar{k}}    \langle \delta_{\bar{k}-1}  ,   x_{\bar{k}-1}  -x \rangle      -   \tfrac{1}{4}  \|   x_{\bar{k}} - x_{\bar{k}-1} \|^2 .
	\label{second_bound}
\end{eqnarray*}
By combining the last two bounds  with \eqref{intermediate} it follows that
\begin{align*}
	\tsum_{t=\bar{k}}^k     \Delta V_t(x)  &\geq \tsum_{t=\bar{k}}^k   \big[  \gamma_t \langle \tilde F(x_{t+1}) , x_{t+1} - x \rangle \big]   - 2 L^2 \gamma_k^2  \|  x - x_{k+1} \|^2   - 2 L^2  \gamma_{\bar{k}}^2  \lambda_{\bar{k}}^2   \| x_{\bar{k}} - x \|^2  \nn \\
&\quad     - \gamma_k \langle \delta_{k+1}  , x_{k+1} - x \rangle + \gamma_k \langle  \delta_k , x_{k} - x \rangle  +   \gamma_{\bar{k}}  \lambda_{\bar{k}}    \langle \delta_{\bar{k}}  ,   x_{\bar{k}}  -x \rangle    - \gamma_{\bar{k}}  \lambda_{\bar{k}}    \langle \delta_{\bar{k}-1}  ,   x_{\bar{k}-1}  -x \rangle   \nn  \\
& \quad    
 - 2 \gamma_{\bar{k}}^2   \lambda_{\bar{k}}^2  \| \delta_{\bar{k}-1} \|_*^2    - 2 \gamma_k^2 \| \delta_k \|_*^2   -  2 \tsum_{t=\bar{k}}^k \left(  \gamma_t^2 \lambda_t^2 \|\delta_t - \delta_{t-1} \|_*^2 \right)        -     \tfrac{ 3}{8} \|  x_{\bar{k}} - x_{{\bar{k}}-1} \|^2  
 + \tfrac{1}{4} \tsum_{t=\bar{k}}^{k-1}  V(x_t, x_{t+1}).
\end{align*}
After rearranging terms
\begin{align*}
& 	2 \tsum_{t=\bar{k}}^k \left(  \gamma_t^2 \lambda_t^2 \|\delta_t - \delta_{t-1} \|_*^2 \right)    +   \big( V(x_{\bar{k}} , x)   +    2 L^2  \gamma_{\bar{k}}^2  \lambda_{\bar{k}}^2   \| x_{\bar{k}} - x \|^2 \big)  - 
 \big( V(x_{k+1} , x)   -  2 L^2 \gamma_k^2  \|  x - x_{k+1} \|^2  \big) \geq   \\
&  \tsum_{t=\bar{k}}^k   \big[  \gamma_t \langle F(x_{t+1}) , x_{t+1} - x \rangle \big] +   \tsum_{t=\bar{k}}^{k}   \big[  \gamma_t   \lambda_t \langle \delta_{t} , x_{t} - x \rangle \big]      +  \gamma_k \langle  \delta_k , x_{k} - x \rangle    - \gamma_{\bar{k}}  \lambda_{\bar{k}}    \langle \delta_{\bar{k}-1}  ,   x_{\bar{k}-1}  -x \rangle   \\
&  - 2 \gamma_{\bar{k}}^2   \lambda_{\bar{k}}^2  \| \delta_{\bar{k}-1} \|_*^2    - 2 \gamma_k^2 \| \delta_k \|_*^2          -     \tfrac{ 3}{8} \|  x_{\bar{k}} - x_{{\bar{k}}-1} \|^2  
+ \tfrac{1}{4} \tsum_{t=\bar{k}}^{k-1}  V(x_t, x_{t+1}).
\end{align*}	
	We now utilize Lemma \ref{auxiliary_sequence}. To this end we consider an auxiliary stochastic process $ \hat{x}_t$, initialized at $ \hat{x}_{\bar{k}-1}= x_{\bar{k}-1}$. We add and subtract the term $ \tsum_{t=\bar{k}}^k \big[  \gamma_t  \langle  \delta_{t} ,  \hat{x}_t  - x \rangle \big]$,  while taking the particular selection of algorithmic parameters into account. Furhtermore we set the stepsize $\hat{\gamma}_t$   to 
	$$
	\hat{\gamma}_t = \begin{cases}
	& \gamma_{\bar{k}} \lambda_{\bar{k}},  \qquad  \qquad t = \bar{k}-1,	 \\
	& \gamma_t  \lambda_t , \qquad   ~\qquad t = \bar{k}, \hdots, k-1, \\
	& \gamma_k \lambda_k + \gamma_k, ~~~~~t = k,
	\end{cases}
	$$ 
to obtain
		\begin{eqnarray*}
&&		2 \tsum_{t=\bar{k}}^k \left(  \gamma_t^2 \lambda_t^2 \|\delta_t - \delta_{t-1} \|_*^2 \right) + 2 V(x_{\bar{k}},x) + V(x_{\bar{k}-1},x) + \tsum_{t=\bar{k}-1}^k  \tfrac{\hat{\gamma}_t^2}{2} \| \delta_t \|^2_*   +     \tfrac{ 3}{8} \|  x_{\bar{k}} - x_{{\bar{k}}-1} \|^2  
+ 2 \gamma_{\bar{k}}^2   \lambda_{\bar{k}}^2  \| \delta_{\bar{k}-1} \|_*^2    + 2 \gamma_k^2 \| \delta_k \|_*^2  
   \geq   \nn  \\
&&  \tsum_{t=\bar{k}}^k \big[  \gamma_t \langle  F(x_{t+1}) , x_{t+1} - x \rangle \big] + \tsum_{t=\bar{k}}^k \big[  \gamma_t \lambda_t \langle  \delta_{t} , x_{t} - \hat{x}_t \rangle \big] +  \gamma_k \langle  \delta_k , x_{k} - \hat{x}_k \rangle    - \gamma_{\bar{k}}  \lambda_{\bar{k}}    \langle \delta_{\bar{k}-1}  ,   x_{\bar{k}-1}  -  \hat{x}_{\bar{k}-1} \rangle .
	\end{eqnarray*}	
	Furthermore by using the monotonicity property in \eqnok{monotone} and 
	the definition of $\bar x_{k+1}$ in \eqnok{def_xbar},
	\begin{eqnarray*}
		&& \tsum_{t=\bar{k}}^k  \big[ \gamma_t \langle F(x_{t+1}) , x_{t+1} - x \rangle \big] \ge   \tsum_{t=\bar{k}}^k  \big[ \gamma_t \langle F(x) , x_{t+1} - x \rangle \big]	
		=(\tsum_{t=\bar{k}}^k   \gamma_t )  \langle F(x), \bar{x}_{k+1} - x  \rangle.
	\end{eqnarray*}
	By employing the definition of the gap function and taking expectations, one obtains \ref{expected_gap_general}.  Moreover 
	by setting 
	$
	\gamma_t = \tfrac{1}{L \sqrt{t}}
	$
	and noticing that  for all $ k \geq 5, $
	$
	\tsum_{t=\bar{k}}^k  \gamma_{t-1}^2  \leq \tfrac{\log(5)}{L^2} ,
	$
	$
	\tsum_{t=\bar{k}}^k  \gamma_{t}  \geq   \tfrac{1}{2 L}   \sqrt{t+1}
	$
	\eqref{expected_gap_fixed} follows.
\end{proof}

In view of proposition \ref{prop_monotone_stohastic},
the  SOE method  finds a solution $\bar x \in X$ s.t. $\mathbb{E}[\gap(\bar x)] \le \epsilon$ in 
${\cal O}(  \tfrac{  1  }{\epsilon^2}) $ iterations.

\section{Stochastic Block OE for Deterministic Problems} \label{sec_SBOE}

In this section, we assume $X \subseteq \mathbb{R}^n$ is a nonempty closed convex set with a block structure, i.e., 
$
X=X_1\times X_2\times \cdot\cdot\cdot\times X_b,
$ 
where $ X_i \subseteq \mathbb{R}^{n_i}, i \in [b],$ are closed convex sets with $n_1+n_2+...+n_b=n.$
We introduce the matrices $U_i \in \mathbb{R}^{n\times n_i}, i \in [b], $ satisfying $(U_1,U_2,...,U_b)=I_n,$ so that the $i$-th block of $x \in X$  is given by 
$ x^{(i)}=U_i^\top x.$  Similarly  given  the operator values  at $x \in X$ we
denote by $F_i(x) \equiv U_i^{\top}F(x),i \in [b]$, the corresponding value associated to the 
$i$-th block. The map $F_i : X \rightarrow \mathbb{R}^{n_i}$ is Lipschitz continuous, i.e.,  for some $L_i>0$, 
\begin{equation}
\|F_i(x_1)-F_i(x_2)\|_{i,*} \leq L_i\|x_1-x_2\|, \quad \forall x_1,x_2\in X.
\end{equation}
Let $\bar{L} = \max_{i=1...b}L_i$,  then  $\forall i \in [b]$,
\begin{equation}
\label{lipschitz_SBOE}
\|F_i(x_1)-F_i(x_2)\|_{i,*} \leq \bar{L}\|x_1-x_2\|, \quad \forall x_1,x_2\in X.
\end{equation}
Each Euclidean space $\mathbb{R}^{n_i}, i\in [b]$, is equipped with inner product $\langle\cdot,\cdot\rangle_i$ and norm $\|\cdot\|_{i}$,  ($\|\cdot\|_{i,*}$ is the conjugate norm). 
We define a norm on $X$ as follows: For $ x = (x^{(1)},  \hdots , x^{(b)}) \in X$,
$
\|x\|^2=\|x^{(1)}\|_1^2+...+\|x^{(b)}\|_b^2,
$
and similarly for the conjugate norm $\|y\|_*^2=\|y^{(1)}\|_{1,*}^2+...+\|y^{(b)}\|_{b,*}^2.$
For a given strongly convex function $\omega_i$ on $X_i$, we define the prox-function (Bregman distance) associated with $\omega_i$ as 
\begin{equation}
\label{block prox function}
V_i(x,y):=\omega_i(x)-\omega_i(y)-\langle\omega_i'(y),x-y\rangle, \quad\forall x,y\in X_i,
\end{equation}
where $\omega_i'(y)\in\partial \omega_i(y)$ is an arbitrary subgradient of $\omega_i$ at $y$. For simplicity, let us use the notation $V_i(x,y) \equiv V_i(x^{(i)},y^{(i)}),  \|x-y\|_i  \equiv\|x^{(i)}-y^{(i)}\|_i   ,  i=1,...,b,$ for $x,y \in X$. Given the individual prox-functions $V_i( \cdot, \cdot)$ on  each $X_i$, $i=1, \hdots, b$,
we define
$
V(x,y):=\tsum_{i=1}^b V_i(x^{(i)},y^{(i)}), ~\forall x,y \in X,
$
and $\omega(x) := \tsum_{i=1}^b\omega_i( x^{(i)}),~\forall x\in X$.

With the above definitions and notation, we can formally describe the stochastic block operator extrapolation algorithm.
\begin{algorithm}[H]  \caption{Stochastic Block Operator Extrapolation (SBOE) for Variational Inequalities}  
	\label{alg:stochastic_block}
	\begin{algorithmic} 
		\STATE{Let $x_0 = x_1 \in X$,  the nonnegative parameters $\{ \gamma_t\}, \{\lambda_t\}$ and the positive probabilities $p_i\in (0,1),i=1,...,b,s.t.\sum_{i=1}^bp_i=1$ be given. }
		\FOR{$ t = 1, \ldots, k$}
		\STATE 1. Generate a random variable $i_t$ according to  
		$
		\mathbb{P}\{i_t=i\}=p_i,\quad i=1,...,b.
		$
		\STATE 2. Update $x_t^{(i)}, i=1,...,b,$ by
		\beq \label{SBOE_step}
		x_{t+1}^{(i)}=\left\{
		\begin{aligned}
			& \mathop{\arg\min}_{u\in X_i}\gamma_t\langle F_i(x_{t})+\lambda_t[F_i(x_{t})-F_i(x_{t-1})],u\rangle+V_i(x_{t},u), & i=i_t; \\
			& x^{(i)}_{t}, & i\neq i_t.
		\end{aligned}
		\right.
		\eeq
		\ENDFOR
	\end{algorithmic}
\end{algorithm} 
In each iteration, Algorithm \ref{alg:stochastic_block} only updates the one block that is randomly chosen. For simplicity, the block to be updated is uniformly chosen in each iteration, i.e. $p_i=1/b, ~\forall i\in[b].$ 
In order to simplify the analysis, we introduce for $ t = 2, \ldots, k+1$ the iterates $\{\hat{x}_t \}$, 
where 
\beq
\label{x_hat}
\hat x_{t+1} = \argmin_{x\in X} ~   \gamma_t \big\langle F(x_t) + \lambda_t \big( F(x_t) - F(x_{t-1}) \big) , x \big\rangle + V(x_t, x).
\eeq


\begin{proposition} \label{prop_SBOE}
	Let $\{x_t\}$ be generated by the SBOE method and $\{\theta_t\} $ be a sequence of nonnegative numbers.
	If the parameters in this method satisfy 
	\begin{align} 
	p_{i}=\tfrac{1}{b},~~i =1, \hdots, b, \label{cond_b}\\
	\theta_{t+1}\gamma_{t+1}\lambda_{t+1}=\theta_t\gamma_t b, \label{cond_lambda_BOE}\\
	\theta_{t-1} \geq 4 \bar{L}^2 \gamma_t^2 \lambda_t^2 \theta_t  \label{cond_16}
	\end{align}
	for all $t = 1, \ldots, k$, then for any $x \in X$,
	\begin{align}\label{prop_general_SBOE_result}
	\tsum_{t=1}^{k-1} \big[\theta_t\gamma_tb-\theta_{t+1}\gamma_{t+1}(b-1) \big]\langle F(x_{t+1}) , x_{t+1} - x \rangle+ \theta_k\gamma_kb\langle F(x_{k+1}) , x_{k+1} - x \rangle+\tsum_{t=1}^{k}\theta_tb V(x_{t+1},x)\nn\\
	- 2\theta_k \gamma_k^2bL^2V(x_{k+1},x)
	+\tsum_{t=1}^k \theta_t\tilde  D_t\leq \tsum_{t=1}^k    \theta_tb V(x_t,x)+\theta_1\gamma_1(b-1)\langle F(x_{1}),x_1-x\rangle,
	\end{align}
	where
	\begin{align}\label{tilde_D}
	\tilde D_t& := \gamma_t \langle F(x_t) + \lambda_t \big( F(x_t)-F(x_{t-1})\big), \hat x_{t+1} - [b x_{t+1}-(b-1) x_t] \rangle\nn\\
	&\quad~ + [V(x_t, \hat x_{t+1}) - b V(x_t,x_{t+1})] + [V(\hat x_{t+1},x) - b V(x_{t+1},x) + (b-1) V(x_{t},x)].
	\end{align}
\end{proposition}

\begin{proof}
	We first note  that, by the update of Algorithm \ref{alg:stochastic_block}, $x_{t+1}$ and $x_t$ only differ in the $i_t$'th block. Using the definition of $\hat x_{t+1}$ in \eqnok{x_hat} and the uniform sampling rule \eqnok{cond_b} we obtain the relations
	\begin{align}\label{unbiasness}
	 x_{t+1}^{(i_t)} = \hat x_{t+1}^{(i_t)}, ~\text{and}~~~\bbe_{i_t}[x_{t+1}] =\tsum_{i=1}^b p_{i}[U_iU_i^\top(\hat x_{t+1} - x_t)+   x_{t}] =  \tfrac{1}{b}\hat x_{t+1} + \tfrac{b-1}{b}x_{t}.
	\end{align}
	By employing Lemma~\ref{lemma_projection}, 
	\begin{align*}
	\gamma_t \langle  F(x_t) + \lambda_t \big( F(x_t) - F(\hat x_{t-1}) \big) , x_{t+1} - x \rangle + V(x_t, \hat x_{t+1}) \leq V(x_t, x) - V(\hat x_{t+1}, x).
	\end{align*}
	Rearraging the terms, we have
	\begin{align}\label{three-point}
	D_t + \tilde D_t + bV(x_t,x_{t+1}) + bV(x_{t+1},x) \leq bV(x_{t},x),
	\end{align}
	where $\tilde D_t$ is defined in \eqnok{tilde_D} and
	\begin{align}\label{D_t}
	D_t  &:= \gamma_t \langle F(x_t) + \lambda_t \big( F(x_t)-F(x_{t-1})\big), b x_{t+1}-(b-1) x_t - x\rangle\nn\\
	&~ = \gamma_t b\langle F(x_{t+1}) , x_{t+1}-x\rangle - \gamma_t (b-1) \langle F(x_t),x_t-x\rangle -\gamma_t b \langle F(x_{t+1})-F(x_t),x_{t+1}-x\rangle\nn\\
	&~ \quad+ \gamma_t\lambda_t \langle F(x_t) - F(x_{t-1}),x_t-x\rangle + \gamma_t\lambda_t b \langle  F(x_{t})-F(x_{t-1}),x_{t+1}-x_t\rangle.
	\end{align}
	Combining \eqnok{three-point} and \eqnok{D_t}, multiplying with $\theta_t$ on both sides and summing up from $t=1$ to $k$, we obtain
	\begin{align*}
	\tsum_{t=1}^k \theta_t b \Delta V_t(x) &\geq \tsum_{t=1}^{k-1} (\theta_t\gamma_t b - \theta_{t+1}\gamma_{t+1}(b-1))\langle F(x_{t+1}) , x_{t+1}-x \rangle -\theta_1 \gamma_1 (b-1)\langle F(x_{1}) , x_{1}-x \rangle\\
	&\quad + \theta_k \gamma_k b  \langle F(x_{k+1}) , x_{k+1}-x \rangle + \tsum_{t=1}^{k-1} (\theta_{t+1} \gamma_{t+1} \lambda_{t+1} - \theta_t \gamma_t b)\langle \Delta F_t,x_{t+1}-x\rangle\\
	 & \quad - \theta_k\gamma_kb  \langle \Delta F_{k+1},x_{k+1}-x\rangle+D +  \tsum_{t=1}^k \theta_t\tilde D_t  ,
	\end{align*}
	where $D: = \tsum_{t=1}^k  \theta_t\gamma_t\lambda_t b \langle  \Delta F_t,x_{t+1}-x_t\rangle+   \tsum_{t=1}^k \theta_t bV(x_t,x_{t+1})$, and $\Delta V_t$ and $\Delta F_t$ are defined in \eqref{def_Delta_V}. Invoking condition \eqnok{cond_lambda_BOE}, we have
	\begin{align}\label{eqn:common_bnd_VI_SBOE}
	\tsum_{t=1}^k \theta_t b \Delta V_t(x) &\geq \tsum_{t=1}^{k-1} (\theta_t\gamma_t b - \theta_{t+1}\gamma_{t+1}(b-1))\langle F(x_{t+1}) , x_{t+1}-x \rangle -\theta_1 \gamma_1 (b-1)\langle F(x_{1}) , x_{1}-x \rangle\nn\\
	&\quad + \theta_k \gamma_k b  \langle F(x_{k+1}) , x_{k+1}-x \rangle  - \theta_k\gamma_kb  \langle \Delta F_{k+1},x_{k+1}-x\rangle+D +  \tsum_{t=1}^k \theta_t\tilde D_t.
	\end{align}
	The strong convexity of the individual prox-functions $V_i$ together with the definition of $V$ gives 
	\begin{align}
	V(x,y) \geq \tfrac{1}{2}\|x-y\|^2, \quad \forall x,y \in X. \label{prop_v}
	\end{align}
	Together with  the Lipschitz condition \eqref{lipschitz_SBOE} and $x_1 = x_0$, we can lower bound the term $D$ as follows:
	\begin{align*}
	D
	&\geq \tsum_{t=1}^k \big[\theta_t\gamma_t b\lambda_t  \langle \Delta F_t, x_{t+1}-x_t\rangle  
	+   \tfrac{b \theta_t}{4} \|  x_t - x_{t+1} \|_{i_t}^2 +   \tfrac{b \theta_{t-1} }{4}\|  x_t - x_{t-1} \|_{i_{t-1}}^2 \big]+       \tfrac{ b}{4}\theta_k \|  x_k - x_{k+1} \|_{i_k}^2  \\
	&\geq \tsum_{t=1}^k  \big[-    \theta_t \gamma_t \lambda_t b \bar{L} \|   x_t - x_{t-1} \|_{i_{t-1}} \|   x_{t+1} - x_t\|_{i_t} 	+
	\tfrac{b\theta_t}{4} \|  x_t - x_{t+1} \|_{i_t}^2 +   \tfrac{b\theta_{t-1} }{4}\|  x_t - x_{t-1} \|_{i_{t-1}}^2 \big]    \\
	&\quad    +\tfrac{ b}{4}\theta_k \|  x_k - x_{k+1} \|_{i_k}^2  \ge   \tfrac{ b}{4}\theta_k \|  x_k - x_{k+1} \|_{i_k}^2.
	\end{align*} 
	Here, the first inequality follows from \eqnok{prop_v}. The second inequality follows from the Lipschitz condition \eqref{lipschitz_SBOE}, the Cauchy-Schwarz inequality, and the relationship $\|x_{t+1}-x_t\|=\|x_{t+1}-x_t\|_{i_t}$. The third inequality follows from \eqnok{cond_16}. Using the above bound on $D$ in \eqnok{eqn:common_bnd_VI_SBOE}, we obtain
	\begin{align}
	\tsum_{t=1}^k    \theta_t b\Delta V_t(x)   &\geq   \tsum_{t=1}^{k-1} \big[\theta_t\gamma_tb-\theta_{t+1}\gamma_{t+1}(b-1) \big]\langle F(x_{t+1}) , x_{t+1} - x \rangle + \theta_k\gamma_kb\langle F(x_{k+1}) , x_{k+1} - x \rangle \nn\\
	& \quad -  \theta_1\gamma_1(b-1)\langle F(x_{1}) , x_{1} - x \rangle- \theta_k \gamma_k b\langle \Delta F_{k+1}, x_{k+1} - x \rangle +       \tfrac{ b}{4}\theta_k \|  x_k - x_{k+1} \|_{i_k}^2+  \tsum_{t=1}^k \theta_t \tilde D_t  \nn\\
	& \geq \tsum_{t=1}^{k-1} \big[\theta_t\gamma_tb-\theta_{t+1}\gamma_{t+1}(b-1) \big]\langle F(x_{t+1}) , x_{t+1} - x \rangle + \theta_k\gamma_kb\langle F(x_{k+1}) , x_{k+1} - x \rangle \nn\\
	& \quad - \theta_1\gamma_1(b-1)\langle F(x_{1}) , x_{1} - x \rangle- 2\theta_k \gamma_k^2L^2V(x_{k+1},x)+  \tsum_{t=1}^k \theta_t \tilde D_t ,
	\label{stochatc_bound_general_SBOE}
	\end{align}
	where the second inequality follows from
	\begin{align*}
	- \gamma_kb \langle \Delta F_{k+1} , x_{k+1} - x \rangle  +
	\tfrac{b }{4} \|  x_k - x_{k+1} \|_{i_k} ^2 &\geq
	-  \gamma_kb L \|  x_k - x_{k+1} \|_{i_k}  \|  x - x_{k+1} \| +
	\tfrac{b}{4} \|  x_k - x_{k+1} \|_{i_k}^2\\
	&  \geq  - b\gamma_k^2L^2  \|  x - x_{k+1} \|^2
	\geq - 2\gamma_k^2bL^2V(x_{k+1},x).
	\end{align*}
	The result is obtained  by rearranging the terms in \eqnok{stochatc_bound_general_SBOE}.
\end{proof}

\subsection{Convergence for GSMVIs with block sturcture}
Given the new Bregman distance $V(\cdot, \cdot)$ the generalized strong monotonicity condition in accordance to \eqnok{G_monotone} still holds.
We now describe the main convergence properties for generalized strongly monotone VIs with block structure.

\begin{theorem} \label{the_SBOE_GMVI}
	Suppose that \eqnok{G_monotone} holds for some $\mu > 0$.
	If the parameters in the SBOE method satisfy 
	\eqnok{cond_b}, \eqnok{cond_lambda_BOE}, \eqnok{cond_16}, 
	\begin{align}
	\theta_{t-1}\gamma_{t-1}b &\geq \theta_t\gamma_t(b-1), \label{lambda_SBOE}\\
	\theta_{t}\big(2\mu \tfrac{b-1}{b}\gamma_t+1\big) &\leq \theta_{t-1} \big(2 \mu \gamma_{t-1} + 1\big), t = 1, \ldots, k, \label{gamma_strong_mon_SBOE} \\
	4 L^2 \gamma_k^2 &\leq1, \label{gamma_strong_mon_SBOE_k}
	\end{align}
	then 
	\begin{align*}
	\theta_k (2 \mu \gamma_k + \tfrac{1}{2})  \bbe[V(x_{k+1},x^*)]
	\leq \theta_1 V(x_1,x^*)  + \tfrac{b-1}{b}\theta_1\gamma_1\langle F(x_{1}) , x_{1} - x^* \rangle.
	\end{align*}
\end{theorem}

\begin{proof}
	Let us fix $x = x^*$ and take expectation on both sides of \eqnok{prop_general_SBOE_result} w.r.t. $ i_1, \hdots, i_k$. Note that $x_t$ is a deterministic function of $ i_1, \hdots, i_{t-1}$. By conditioning on $ i_1, \hdots, i_{t-1}$ and using the law of iterated expectations and \eqnok{unbiasness}, it follows that 
	$
	\mathbb{E}[ \tilde D_t ] = 0.
	$ 
	As such
	\begin{align*}
	\tsum_{t=1}^{k-1} \big[\theta_t\gamma_tb-\theta_{t+1}\gamma_{t+1}(b-1) \big]\bbe[\langle F(x_{t+1}) , x_{t+1} - x^* \rangle]+ \theta_k\gamma_kb\bbe[\langle F(x_{k+1}) , x_{k+1} - x^* \rangle]\nn\\
	+\tsum_{t=1}^{k}\theta_tb\bbe[V(x_{t+1},x^*)]- 2\theta_k\gamma_k^2bL^2\bbe[V(x_{k+1},x^*)]
	\leq \tsum_{t=1}^k    \theta_tb \bbe[V(x_t,x^*)]+\theta_1\gamma_1(b-1)\langle F(x_{1}),x_1-x^*\rangle.\label{bnd_SBOE_temp}
	\end{align*}
	Using the above inequality,  \eqnok{G_monotone}, and \eqnok{lambda_SBOE}, we obtain
	\begin{align*}
	\tsum_{t=1}^{k-1} 2\mu\big[\theta_t\gamma_tb-\theta_{t+1}\gamma_{t+1}(b-1) \big]\bbe[V(x_{t+1},x^*)]+ 2\mu b\theta_k\gamma_k\bbe[V(x_{k+1},x^*)]+\tsum_{t=1}^{k}\theta_tb\bbe[V(x_{t+1},x^*)]\nn\\
	- 2\theta_k\gamma_k^2bL^2\bbe[V(x_{k+1},x^*)]
	\leq \tsum_{t=1}^k    \theta_t b\bbe[V(x_t,x^*)]+\theta_1\gamma_1(b-1)\langle F(x_{1}),x_1-x^*\rangle.
	\end{align*}
	Invoking \eqnok{gamma_strong_mon_SBOE}, we have
	\begin{align*}
	\theta_k (2 \mu \gamma_k + 1)  \bbe[V(x_{k+1},x^*)] - 2\theta_k\gamma_k^2L^2\bbe[V(x_{k+1},x^*)]
	\leq \theta_1 V(x_1,x^*)  + \tfrac{b-1}{b}\theta_1\gamma_1\langle F(x_{1}) , x_{1} - x^* \rangle,
	\end{align*}
	which, in view of \eqnok{gamma_strong_mon_SBOE_k},  implies the result.
\end{proof}

We now specify the selection of a particular
stepsize policy for solving GSMVI problems with block structure. We show that by using a constant stepsize $\gamma_t=\gamma_0$, one can achieve a linear convergence rate.

\begin{corollary}\label{corrollary_SBOE}
	If
	$$
	\gamma_t =\gamma=\tfrac{1}{2\bar{L}b}, ~~~ \lambda_t=\tfrac{b+2(b-1)\mu\gamma}{1+2\mu\gamma},~~\mbox{and}~~ \theta_t  = \big(\tfrac{1+2\mu\gamma}{1+2\mu\gamma(b-1)/b}\big)^k,
	$$
	then
	$
	\bbe[V(x_{k+1},x^*)] 
	\leq 2\big(\tfrac{1+2\mu\frac{b-1}{b}\gamma}{1+2\mu\gamma}\big)^k\big[V(x_1,x^*)+\tfrac{b-1}{b}\gamma\langle F(x_{1}),x_1-x^*\rangle\big].
	$
\end{corollary}

\begin{proof}
	Note that \eqnok{cond_lambda_BOE} holds by the definition of $\lambda_t$.
	Observe that
	\begin{align*}
	4 L^2 \gamma_t^2 \leq  \tfrac{L^2}{\bar{L}b^2} \le1, ~~~~4 \bar{L}^2 \gamma_t^2 \lambda_t^2 \theta_t = 4\bar{L}^2b^2 \gamma_t^2 \tfrac{\theta_{t-1}^2 \gamma_{t-1}^2}{\theta_t^2 ~  \gamma_t^2} \theta_t
	= 4\bar{L}^2b^2  \tfrac{\theta_{t-1}^2 \gamma_{t-1}^2}{\theta_t} \le 4\bar{L}^2b^2  \gamma_{t-1}^2 \theta_{t-1} \le \theta_{t-1},
	\end{align*}
	and thus both  \eqnok{cond_16} and \eqnok{gamma_strong_mon_SBOE_k} hold.
	In order to check \eqnok{gamma_strong_mon_SBOE}, we observe that
	$
	\theta_{t-1} (2 \mu \gamma_{t-1} + 1) = \tfrac{\theta_t \lambda_t}{b}(2 \mu \gamma_{t-1} + 1)= \theta_{t}\big(2\mu \tfrac{b-1}{b}\gamma_t+1\big).
	$
	From the definition of $\gamma_t$ and $\theta_t$, \eqnok{lambda_SBOE} is satisfied.
	The result then follows from Theorem~\ref{the_SBOE_GMVI}.
\end{proof}

For the specific choice of algorithmic parameters 
$\{ \gamma_t\}, \{\lambda_t\}$, and $\{\theta_t\}$ employed in Corollary~\ref{corrollary_SBOE}, the total number of iterations required by the SBOE algorithm to 
find an $\epsilon$-solution can be bounded by
$
\mathop{O}\{b^2\bar{L}/\mu\log{(1/\epsilon)\big\}}.
$
This result shows that we can obtain linear rate of convergence for solving VI problems even though the algorithm is stochastic. Comparing with the deterministic OE algorithm, the iteration cost of SBOE is cheaper up to a factor of $\mathop{O}\{b\}$. Note that  the Lipschitz constant is $\bar{L}=\mathop{O}\{L/\sqrt{b}\}$.
 In terms of the dependence on $b$, the SBOE algorithm is $\mathop{O}\{\sqrt{b}\}$ worse, while it is not evident to us that this dependence can be improved.

\subsection{Convergence for MVIs with block structure}
In this subsection we show that  the SBOE method 
can achieve a $\mathop{O}(\tfrac{1}{\epsilon})$ convergence rate for solving standard monotone VIs. 
In the following we assume the norm $\|\cdot\|$ is Euclidean and $\o(x)=\tfrac{\|x\|^2}{2}, \forall x \in X$. \textcolor{black}{Furthermore we assume that the set $X$ is bounded.}
The proof of this result relies on the construction of a novel auxiliary sequence $\{ z_t\}$ that is related but different from
those used in the VI literature~\cite{NJLS09-1,AlacaogluFercoqCevher2020}.

\begin{theorem}\label{MVI_SBOE}
	Suppose that \eqnok{monotone} holds. Let $\{x_t\}$ be generated by Algorithm \ref{alg:stochastic_block} and
	denote
	\beq \label{def_xbar1}
	\bar{x}_{k+1} := \tfrac{\sum_{t=1}^{k-1} [\theta_t\gamma_t b -\theta_{t+1}\gamma_{t+1} (b-1)]x_{t+1} +\theta_k\gamma_k b x_{k+1}}{\sum_{t=1}^{k-1} [\theta_t\gamma_t b -\theta_{t+1}\gamma_{t+1} (b-1)] +\theta_k\gamma_k b }.
	\eeq
	If \eqref{cond_b}, \eqref{cond_lambda_BOE}, \eqref{lambda_SBOE}, \eqref{gamma_strong_mon_SBOE_k} hold, and $\theta_{t-1} \geq 16 \bar{L}^2 \gamma_t^2 \lambda_t^2 \theta_t$ ,$\theta_t \leq \theta_{t-1}$, then
	\beq \label{result_1_SBOEMVI}
	\bbe[\gap(\bar{x}_{k+1})]
	\leq  \tfrac{\theta_1}{\sum_{t=1}^{k-1} [\theta_t\gamma_t b -\theta_{t+1}\gamma_{t+1} (b-1)] +\theta_k\gamma_k b} \max_{x \in X} \big[5(b+1)V(x_1,x) +\gamma_1 (b-1)\langle F(x_1),x_1-x \rangle \big].
	\eeq
	In particular, if $\gamma_t = \tfrac{1}{4\bar L b}$, $\lambda_t = b$ and $\theta_t =  1$,
	then
	\beq \label{result_2_SBOEMVI}
	\bbe[\gap(\bar{x}_{k+1}) ]
	\leq  \tfrac{4\bar L b}{k-1+b} \max_{x \in X} \big[5(b+1)V(x_1,x) +\tfrac{b-1}{4\bar L b}\langle F(x_1),x_1-x \rangle \big].
	\eeq
\end{theorem}
\begin{proof}
Note that \eqnok{eqn:common_bnd_VI_SBOE} still holds, and therefore
\begin{align}\label{equ1}
\tsum_{t=1}^k \theta_t b \Delta V_t(x) \geq S +  \tsum_{t=1}^k \theta_t (A_t + B_t) + C,
\end{align}
where
\begin{align*}
	S &:= \tsum_{t=1}^{k-1} (\theta_t\gamma_t b - \theta_{t+1}\gamma_{t+1}(b-1))\langle F(x_{t+1}) , x_{t+1}-x \rangle -\theta_1 \gamma_1 (b-1)\langle F(x_{1}) , x_{1}-x \rangle \\
	&\quad ~ + \theta_k \gamma_k b  \langle F(x_{k+1}) , x_{k+1}-x \rangle  - \theta_k\gamma_kb  \langle \Delta F_{k+1},x_{k+1}-x\rangle\\
	A_t &:= \gamma_t \langle F(x_t) + \lambda_t \big( F(x_t)-F(x_{t-1})\big), \hat x_{t+1} - [b x_{t+1}-(b-1) x_t] \rangle + [V(x_t, \hat x_{t+1}) - b V(x_t,x_{t+1}) \\
	B_t &:= V(\hat x_{t+1},x) - b V(x_{t+1},x) + (b-1) V(x_{t},x), \\
	C & := \tsum_{t=1}^k  \theta_t\gamma_t\lambda_t b \langle  \Delta F_t,x_{t+1}-x_t\rangle+   \tsum_{t=1}^k \theta_t bV(x_t,x_{t+1}).
\end{align*}
By the definition of Bregman distance, we have
$
B_t  = B_t^{(1)} + B_t^{(2)} + \langle \hat \delta_t, x \rangle,
$
where 
\begin{align*}
B_t^{(1)}&= - \omega(\hat x_{t+1})+ b \omega(x_{t+1})-(b-1)\o(x_t),\\
B_t^{(2)}&= \langle \nabla \o(\hat x_{t+1}), \hat x_{t+1}\rangle - b \langle\nabla\o(x_{t+1}),x_{t+1} \rangle + (b-1) \langle \nabla\o(x_{t}),x_{t}  \rangle, \\
\hat \delta_t &= b\nabla \o(x_{t+1})-(b-1)\nabla \o(x_t)-\nabla \o(\hat x_{t+1}).
\end{align*}
We now introduce an auxiliary sequence $\{z_t\}$, which satisfies $z_{1}=x_1$ and for $t = 1,2, \hdots, k, $
 $$z_{t+1} = \arg \min_{x \in X} \langle \hat\delta_t, x \rangle + 4(b+1)
V({z}_t,x) \textcolor{black}{.}$$ 
Utilizing Lemma \ref{auxiliary_sequence} and $\theta_t \leq \theta_{t-1}$, we obtain
\begin{align*}
 \tsum_{t=1}^k  \theta_t \langle \hat \delta_t , z_t - x \rangle \leq 4\theta_1(b+1) V(z_{1},x)+
\tsum_{t=1}^k  \tfrac{\theta_t\| \hat \delta_t \|^2_*}{8(b+1)} \leq 4\theta_1(b+1) V(z_{1},x)+
\tsum_{t=1}^k  \tfrac{\theta_t(\| x_t - \hat x_{t+1} \|^2+b^2\| x_t - x_{t+1} \|^2)}{4(b+1)} .
\end{align*}
The second inequality follows from $\o(\cdot)=\tfrac{\|\cdot\|^2}{2}$ and Young's inequality. Combining  the relation above with \eqnok{equ1} gives us
\begin{align}\label{equ2}
\tsum_{t=1}^k \theta_t b \Delta V_t(x) &\geq S+  \tsum_{t=1}^k \theta_t (A_t + B_t^{(1)}+B_t^{(2)}+\langle \hat \delta_t , z_t  \rangle)  + C- 4\theta_1(b+1) V(z_{1},x)\nn\\
&\quad -\tsum_{t=1}^k \tfrac{\theta_t(\| x_t - \hat x_{t+1} \|^2+b^2\| x_t - x_{t+1} \|^2)}{4(b+1)}.
\end{align}
Meanwhile, we have 
\begin{align*}
C-&\tsum_{t=1}^k \tfrac{\theta_t(\| x_t - \hat x_{t+1} \|^2+b^2\| x_t - x_{t+1} \|^2)}{4(b+1)}\\&\geq \tsum_{t=1}^k  \theta_t\gamma_t\lambda_t b \langle \Delta F_t,x_{t+1}-x_t\rangle+   \tsum_{t=1}^k  \tfrac{\theta_t b\|x_t-x_{t+1}\|^2}{4} +   \tsum_{t=1}^k  \big[\tfrac{\theta_t b\|x_t-x_{t+1}\|^2}{4} - \tfrac{\theta_t(\| x_t - \hat x_{t+1} \|^2+b^2\| x_t - x_{t+1} \|^2)}{4(b+1)}\big]\\
&\geq \tsum_{t=1}^k  \big[-    \theta_t \gamma_t \lambda_t b \bar{L} \|   x_t - x_{t-1} \| \|   x_{t+1} - x_t\|	+
\tfrac{b\theta_t}{8} \|  x_t - x_{t+1} \|^2 +   \tfrac{b\theta_{t-1} }{8}\|  x_t - x_{t-1} \|^2 \big]    +\tfrac{ b}{8}\theta_k \|  x_k - x_{k+1} \|^2 \\
&\quad   +   \tsum_{t=1}^k \tfrac{\theta_t(b\| x_t - x_{t+1} \|^2-\| x_t - \hat x_{t+1} \|^2)}{4(b+1)}  \\
&\ge   \tfrac{ b}{8}\theta_k \|  x_k - x_{k+1} \|^2 +   \tsum_{t=1}^k \tfrac{\theta_t(b\| x_t - x_{t+1} \|^2-\| x_t - \hat x_{t+1} \|^2)}{4(b+1)} .
\end{align*} 
Here the first inequality follows from \eqref{prop_v}, the second inequality follows from \eqref{lipschitz_SBOE}, and the third inequality follows from the condition $\theta_{t-1} \geq 16 \bar{L}^2 \gamma_t^2 \lambda_t^2 \theta_t$. 
Using the above bound in \eqnok{equ1}, the definition of $S$ and Young's inequality, we obtain
\begin{align*}
&\tsum_{t=1}^k    \theta_t b\Delta V_t(x)   \nn\\
&\geq S + \tfrac{b}{8} \theta_k \|x_k - x_{k+1}\|^2 +\tsum_{t=1}^k\theta_t \left(A_t + B_t^{(1)}+B_t^{(2)}+\langle \hat \delta_t , z_t  \rangle+\tfrac{b\| x_t - x_{t+1} \|^2-\| x_t -\hat x_{t+1} \|^2}{4(b+1)}\right) - 4\theta_1(b+1) V(x_{1},x)\nn\\
& \geq \tsum_{t=1}^{k-1} \big(\theta_t\gamma_tb-\theta_{t+1}\gamma_{t+1}(b-1) \big)\langle F(x_{t+1}) , x_{t+1} - x \rangle - \theta_1\gamma_1(b-1)\langle F(x_{1}) , x_{1} - x \rangle + \theta_k\gamma_kb\langle F(x_{k+1}) , x_{k+1} - x \rangle \nn\\
& \quad - 2\theta_k \gamma_k^2L^2\|x_{k+1}-x\|^2 +\tsum_{t=1}^k\theta_t \left(A_t + B_t^{(1)}+B_t^{(2)}+\langle \hat \delta_t , z_t  \rangle+\tfrac{b\| x_t - x_{t+1} \|^2-\| x_t - \hat x_{t+1} \|^2}{4(b+1)}\right)- 4\theta_1(b+1) V(x_{1},x).
\end{align*}
After taking expectation on both sides of the inequality, it follows by \eqnok{unbiasness} that $\bbe[A_t + B_t^{(1)}+B_t^{(2)}+\theta_t \langle \hat \delta_t , z_t  \rangle+\theta_t\tfrac{b\| x_t - x_{t+1} \|^2-\| x_t - \hat x_{t+1} \|^2}{4(b+1)}]=0$. Together with the condition $\theta_t \leq \theta_{t-1}$ and \eqnok{gamma_strong_mon_SBOE_k} we obtain
\begin{align*}
\tsum_{t=1}^{k-1} \big[\theta_t\gamma_tb-\theta_{t+1}\gamma_{t+1}(b-1) \big]\bbe[\langle F(x_{t+1}) , x_{t+1} - x \rangle] + \theta_k\gamma_kb\bbe[\langle F(x_{k+1}) , x_{k+1} - x \rangle] \\
\leq \theta_1 b V(x_1,x) + 4\theta_1(b+1) V(x_{1},x) + \theta_1\gamma_1(b-1)\langle F(x_{1}) , x_{1} - x \rangle.
\end{align*}
Furthermore, similar as in the proof of Theorem \ref{the_monotone}, using the monotonicity property in \eqnok{monotone} and 
the definition of $\bar x_{k+1}$, we get \eqnok{result_1_SBOEMVI}.
\end{proof} 

In view of Theorem~\ref{MVI_SBOE},
the SBOE method can achieve a ${\cal O}(b^2 \bar L/\epsilon)$ complexity rate for MVI problems. 
Similarly to the GSMVI case,
 the SBOE algorithm is $\mathop{O}\{\sqrt{b}\}$ worse than the OE algorithm in terms of the dependence on number of blocks.

\section{Numerical experiments} \label{sec-num}   
In this section we report some preliminary numerical results for the operator extrapolation (OE) algorithm as well as its stochastic (SOE) and stochastic block (SBOE) variants.
\subsection{Traffic assignment problem} \label{transportation}   
We consider a classic problem in operations research, the general traffic assignment problem, where the travel cost on each link in the transportation network may depend on the flow on this link as well as other links in the network. 
Algorithmic design for the computation of  traffic equilibrium patterns based on the theory   of variational inequalities originated in 
the work of \cite{Dafermos80}.
The traffic network is abstracted as a directed graph, consisting of a set of nodes $ \mathcal{N}$, a set of directed arcs $ \mathcal{A}$, $| \mathcal{A} | = n$,
together with a set of ordered node pairs
$\mathcal{W}$, $| \mathcal{W} | = N$, where an element $ w = (o,d) \in \mathcal{W}$ is referred to as an origin-destination (OD) pair. The travel demand $d_w \geq 0 $ associated to the OD pair $w \in \mathcal{W}$ is to be distributed among the paths of the network that connect $w$, the latter set is denoted by $P_w$. The set of feasible path flow vectors is $X \subset \mathbb{R}^N$,
$
X = \{ x \in \mathbb{R}^N~|~\forall w \in \mathcal{W}, \forall p \in P_w, \tsum_{p \in P_w} x_p = d_w,  ~ x_p \geq 0   \}.
$

Let $A$ denote the arc-chain incidence matrix of the network.
A path flow pattern $x \in X$ induces an arc flow pattern $y \in \mathbb{R}^n$ and the set of feasible arc flows is 
$
Y = \{ y \in \mathbb{R}^n ~|~   y = A x, x \in X  \}.
$
We introduce the map $F : \mathbb{R}^n \rightarrow \mathbb{R}^n $, where $F(y)$ denotes the vector of costs associated to a particular arc flow pattern $y$. 
The traffic assignment problem admits a VI formulation:
Given a transportation network and a demand vector $d \in \mathbb{R}^{|\mathcal{W}|}$, 
\beq
\label{VIP_TA}
\text{find}~ y^* \in Y: ~~~~ \langle F(y^*) , y - y^* \rangle \geq 0, ~~~ \forall y \in Y.
\eeq
The equilibrium condition \eqref{VIP_TA} is referred in the literature as the user-optimization principle.
We consider an affine parameterization for the operator $F$, with 
$
F(y) = G y + b,~ G \in \mathbb{R}^{n \times n}_+, ~\mbox{and}~ b \in \mathbb{R}^n_+. 
$
The matrix $G$ is not necessarily symmetric. We test our algorithms on randomly generated data sets, where  $A= I$, $ |\mathcal{W} |= 5$, $ G + G^{\T} \succ 0$,
and $b=5$, while the number of arcs  is progressively increased.  \textcolor{black}{Given that the operator $F$ is affine,  one can directly compute
$ L =  \sigma_{\max}(G)$, $\mu = \tfrac{1}{2} \lambda_{\min}( G + G^{\T})$.}
The algorithm proposed in \cite{NesterovScimali06} abbreviated as (NS) is also considered for the purposes of illustration.  

\begin{figure}[H]
	\centering
	\begin{minipage}[t]{0.4\linewidth}
		\centering
		\includegraphics[width=6cm]{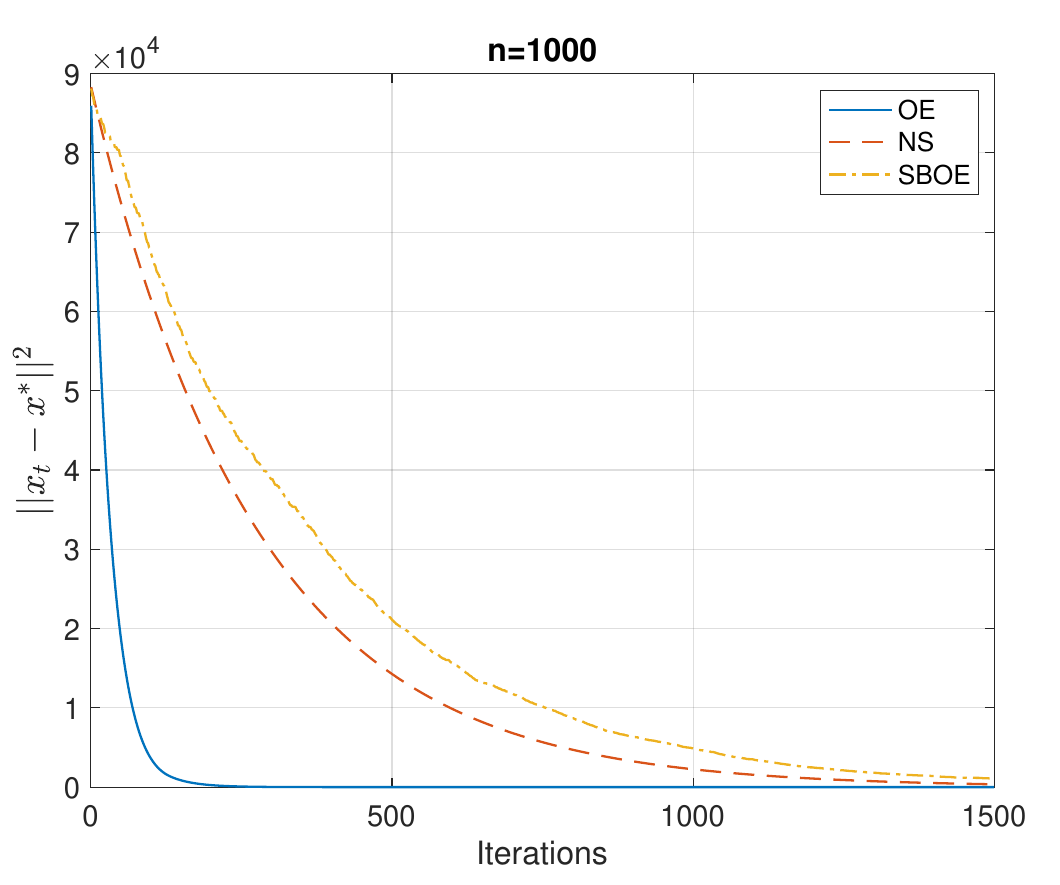}
	\end{minipage}
	\begin{minipage}[t]{0.4\linewidth}
		\centering
		\includegraphics[width=6cm]{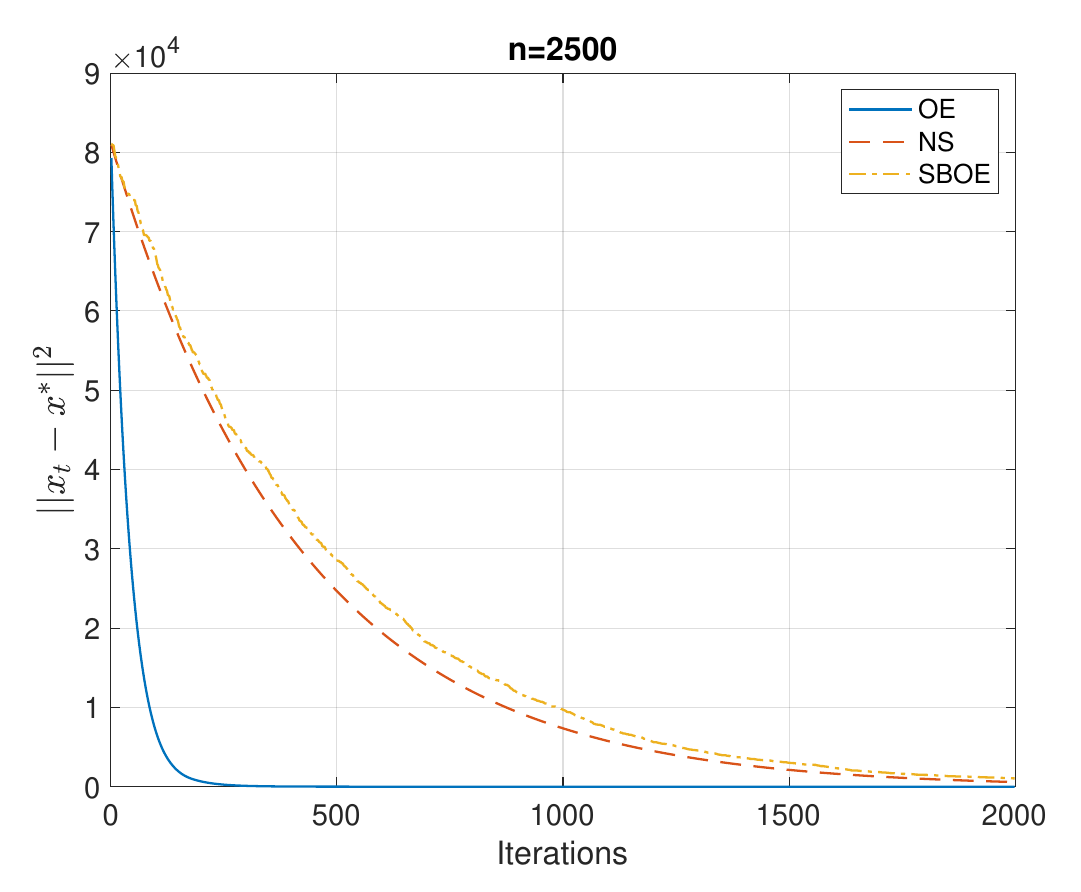}
	\end{minipage}
	\begin{minipage}[t]{0.4\linewidth}
		\centering
		\includegraphics[width=6cm]{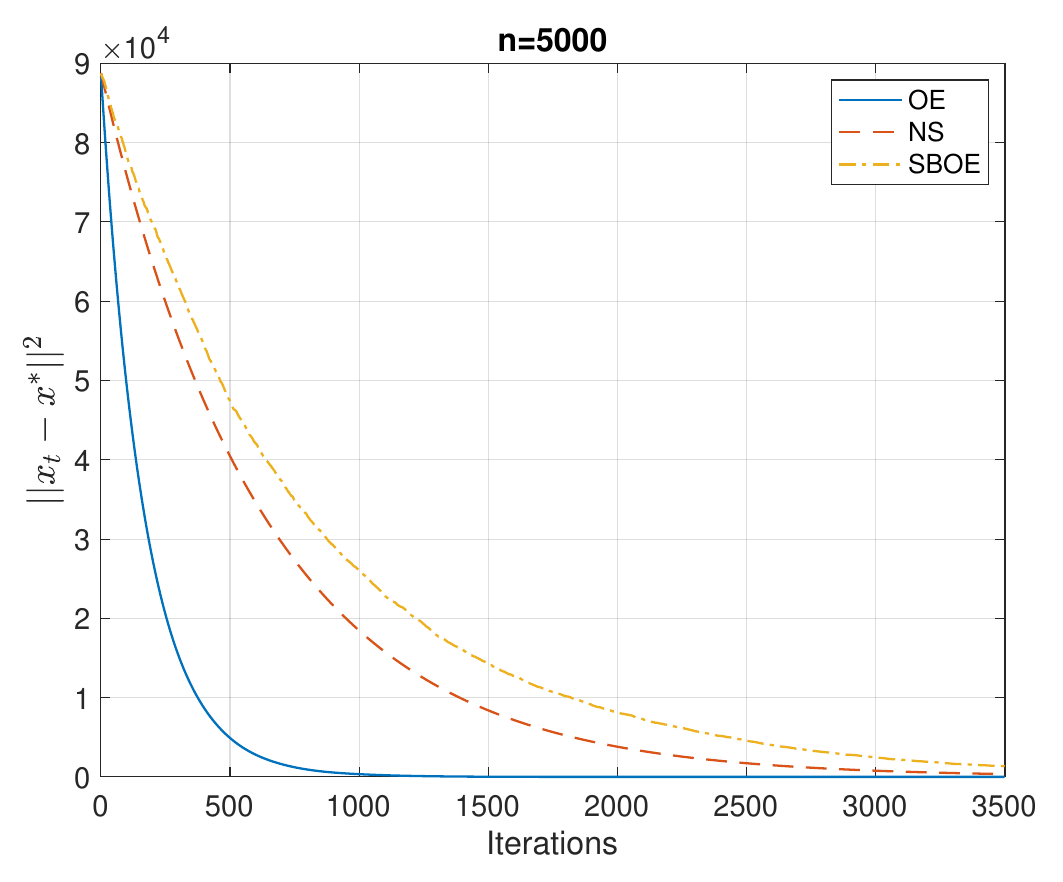}
	\end{minipage}
	\begin{minipage}[t]{0.4\linewidth}
		\centering
		\includegraphics[width=6cm]{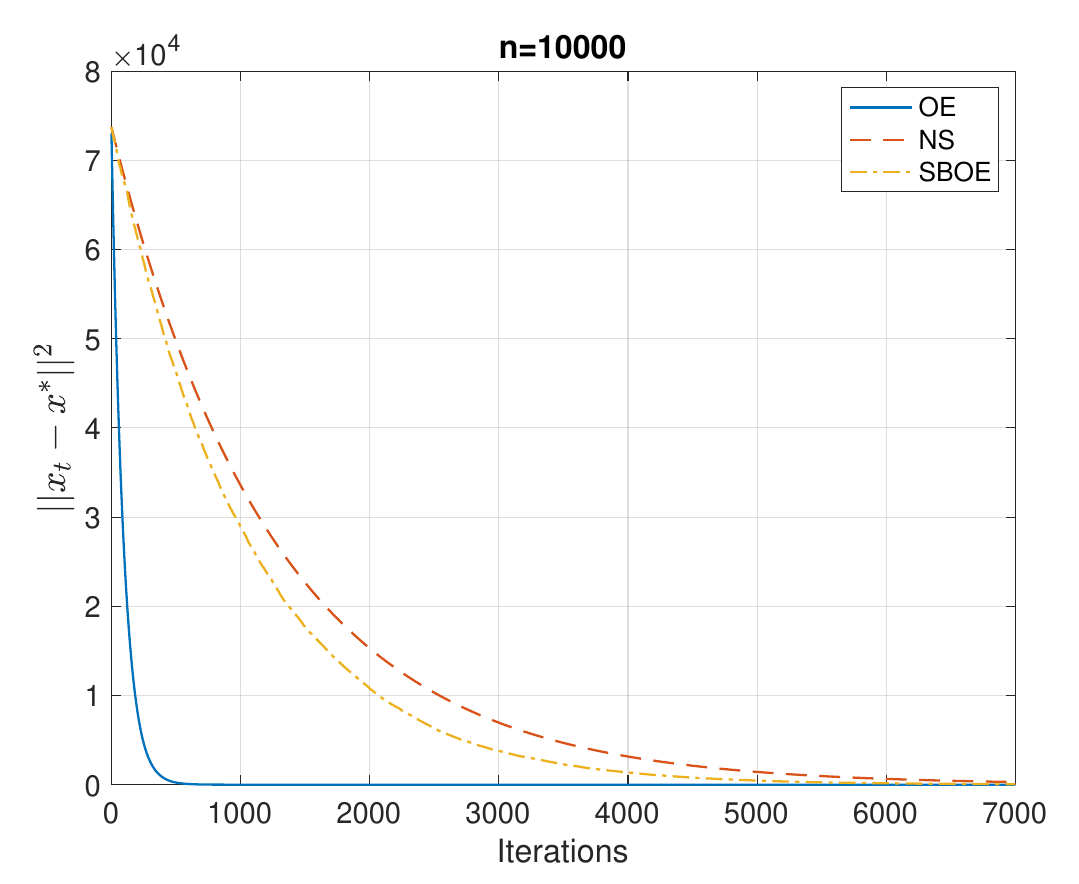}
	\end{minipage}
	\caption{Traffic assignment problem for networks with \textcolor{black}{1000,  2500, 5000, 10000} arcs. Comparison of error trajectory $ \|y_t - y^*\|^2 = \|x_t - x^*\|^2 $  for OE, NS and SBOE algorithm, number of blocks is $b=5$.}
	\label{fig_compare}
\end{figure}.  

\begin{table}[H]  
	\centering
	\begin{tabular}{|c|c|c|c|c|c|c|}
		\hline
		$n$& $L$&$\mu$& $ \frac{L}{\mu}$&OE&NS&SBOE\\
		\hline
		$1000$& $72.02$& $0.134$& $ 537.70$& $0.765$&$1.517$&$0.407$\\
	\hline
$2500$& $112.03$&$0.133$& $ 841.99$ &$0.934$ & $1.789$ &$0.680$\\
		\hline
$5000$& $162.14$&$0.129$& $ 1259.40$ &$1.619$ & $2.658$ &$0.755$\\
		\hline
$10000$& $237.18$&$0.094$& $ 2532.40$ &$2.701$ & $4.687$ &$1.086$\\
		\hline
	\end{tabular}
	\caption{Comparison of average CPU time per iteration in [sec], last three columns}\label{table1}
\end{table}
\textcolor{black}{
In all four experiments the (OE) algorithm seems to exhibit faster convergence to the equilibirum point in terms of the iteration count.
The (NS) algorithm requires the maintenance of two sequences of iterates which is reflected in approximately double the average computation time per iteration in comparison to the (OE) algorithm.
As expected the (SBOE) algorithm is the least computationally intensive given that it has the advantage of solving a lower dimensional optimization problem at each iteration, while it can also benefit from a recursive update of the operator utilizing the underlying block structure, see also \cite{LanBook2020}.}

\subsection{Signal estimation and generalized linear models} \label{arkadi examples} 

We consider a nonlinear signal estimation problem involving generalized linear models (GLMs) that is amenable to a variational inequality formulation
(see Chapter 5.2 of  \cite{ArkadiBook2020}). 
An i.i.d. sequence of regressor-label observations 
$$
\xi^K = \{  \xi_k = (\eta_k, y_k), 1 \leq k \leq K  \}
$$
is generated according to a distribution 
$P_{x^*}$, where $x^* \in \mathcal{X} \subset 
\mathbb{R}^n$ is an unknown signal lying in a convex compact set. 
The regressor vector $ \eta $ has distribution $Q$ independent of $x^*$, while the conditional distribution of the label $y$ given $\eta$ is controlled by $x^*$ and
$
\mathbb{E}[y | \eta] = f(\eta^{\T} A x^*),
$
where  $f: \mathbb{R} \rightarrow \mathbb{R}$ is referred to in the statistics literature as the link function and $A \in \mathbb{R}^{n \times n}$ is a full rank, known matrix, that satisfies the condition
$
A + A^{\T} \succ 0.
$
The goal is to infer $x^* \in \mathcal{X}$ from the i.i.d. realizations $\xi^K$ generated by the unknown distribution $P_{x^*}$. This inference problem admits the following VI formulation.  
Let the operator 
$
F : \mathcal{X} \rightarrow \mathcal{X}
$
be defined as 
$$
F(x) = \mathbb{E}[\eta f(\eta^{\T} A x) ]-
\mathbb{E}[\eta f(\eta^{\T} A x^*)].
$$ 
Then $x^*$ is a zero of $F$ and under mild regularity assumptions
is given as the solution of the VI problem of finding $x^* \in X$
s.t.  $\langle F(x^*) , x - x^* \rangle \geq 0$ for all $x \in X$.
An unbiased estimator of $F(x)$ is given by
$
\tilde{F}(x) = \eta f(\eta^{\T} A x )- \eta y. 
$
We will apply the four versions of the (SOE) algorithm with the stepsize selection as in 
Corollary \ref{step_size_stoch_strong_mon_k_unknown} (SOE-1),  
Corollary \ref{step_size_stoch_strong_mon_k_known} (SOE-2),
Corollary \ref{restart_stepsize} (SOE-3)
and Theorem \ref{lemma_dist_no_strong_monotone_SOE} (SOE-4)
and compare to the classic stochastic approximation (SA) algorithm. 
We consider the case where $ \mathcal{X} = \{ x \in \mathbb{R}^n ~|~ \| x \| \leq R\}$, is a ball of radius $R$, while the distribution $Q$ of the regressor $\eta $ is $ \mathcal{N}(0, I_n)$.
There are two options for the link function:
\begin{table}[H]  
	\centering
	\begin{tabular}{|c|c|c|c|c|c|}
		\hline
		A& hinge function & $y \in \mathbb{R}$ & $f(s) = \max\{s,0\}$ \\
		\hline	
		B& ramp sigmoid & $y \in \mathbb{R}$ & $f(s) = \min\{1, \max\{s,0\} \}$ \\
		\hline	
	\end{tabular}
	\caption{link functions}\label{table2}
\end{table}
In both cases the conditional distribution of the label $y$ given the regressor $\eta$ is   Gaussian $ \mathcal{N}(f(\eta^\T A x^*), \sigma_y)$, while
the operator $F$ is Lipschitz continuous and strongly monotone. The following observation 
facilitates the analytic calculation of bounds to the Lipschitz constant as well as the modulus of strong monotonicity in each case.
Fix $ x \in \mathbb{R}^n - \{0\}$,  and set $ z = A x$. Consider the operator 
$
G(x) = \mathbb{E}[ \eta f( \eta^{\T} A x)]$,  and note that $ F(x) = G(x) - G(x^*).$
Set 
$u_1 = \tfrac{z}{\|z \|}$ and extend this vector  to an orthonormal basis $ \{ u_1, \hdots, u_n\}$ of $\mathbb{R}^n$. The Gaussian vector $ \eta $   is written as $ \eta = \tsum_{1}^{n} \langle u_i , \eta \rangle u_i $. One observes that 
$ \langle \eta, z \rangle $ and $ \eta_{\bot} = \tsum_{2}^{n} \langle u_i , \eta \rangle u_i $ are independent, since for any $j \in \{2, n\}$ the normal r.v.s 
$ 
\langle \eta, z \rangle,  \langle \eta, u_j \rangle
$
are uncorrelated.  This realization allows one to  write
$$
G(x) =  \mathbb{E}[ \eta f( \eta^{\T} A x)] = \mathbb{E}[\tfrac{z z^{\T}}{\|z\|^2} \eta f( \eta^{\T} z)]
=  \tfrac{z }{\|z\|} \mathbb{E}[\zeta  f( \zeta \|z\|)],
$$
where $ \zeta = \eta^{\T}  \tfrac{z }{\|z\|} $ is a standard normal.

\subsubsection{Hinge function}
In this case $ G_A(x) = \tfrac{1}{2} A x$ and as  such $ L_A = \tfrac{1}{2} \sigma_{\max}(A)$, $\mu_A = \tfrac{1}{4} \lambda_{\min}( A + A^{\T})$.  It is worth noting that in this example, while the link function is non-linear the corresponding operator is linear proportional to the one of the linear regression case, where the link function is $ f(s) = s$.
We test the algorithms on some randomly
generated data sets with dimension $n=100$ and radius $R=100$. The vector $x^*$ is created by sampling each of
its entries uniformly in $[0,1]$ and subsequently normalizing such that $ \| x^* \| = R$. The matrix $A$ is of the form $ A = \diag{(d)} + d_{-} \times 10^{-2} \hat{A}$ where the entries of 
$  \hat{A}$ are sampled uniformly in $[0,1]$ and the vector $d \in \mathbb{R}^{100}$ has entries
chosen equidistantly between $ d_{-} $ and $ d_{+} = 1 $. The parameter $d_{-} $ is progressively decreased in order to achieve lower $ \mu$ and therefore higher condition numbers.  

 The theoretical analysis provides us with conservative stepsize policies which will ensure the convergence for all the algorithms above. In terms of the actual implementation we fine-tuned each stepsize policy, by changing the value of the  Lipschitz constant $L$.
For fairness purposes we maintain the same $L$  across all  algorithms and stepsize policies.  We fine-tuned the value of $L$ based on the first 50 iterations and set it to 
$L = 0.5$.
Our choice of $L$  improves the convergence speed while ensuring that the error of the algorithm decreases steadily. The motivation behind our approach lies in the fact that 
the bound of \eqref{split_8} only requires a local Lipschitz constant for each time step $t$, which can be  smaller than the global one. 

\begin{figure}[H]
	\centering
	\begin{minipage}[t]{0.3\linewidth}
		\centering
		\includegraphics[width=5cm]{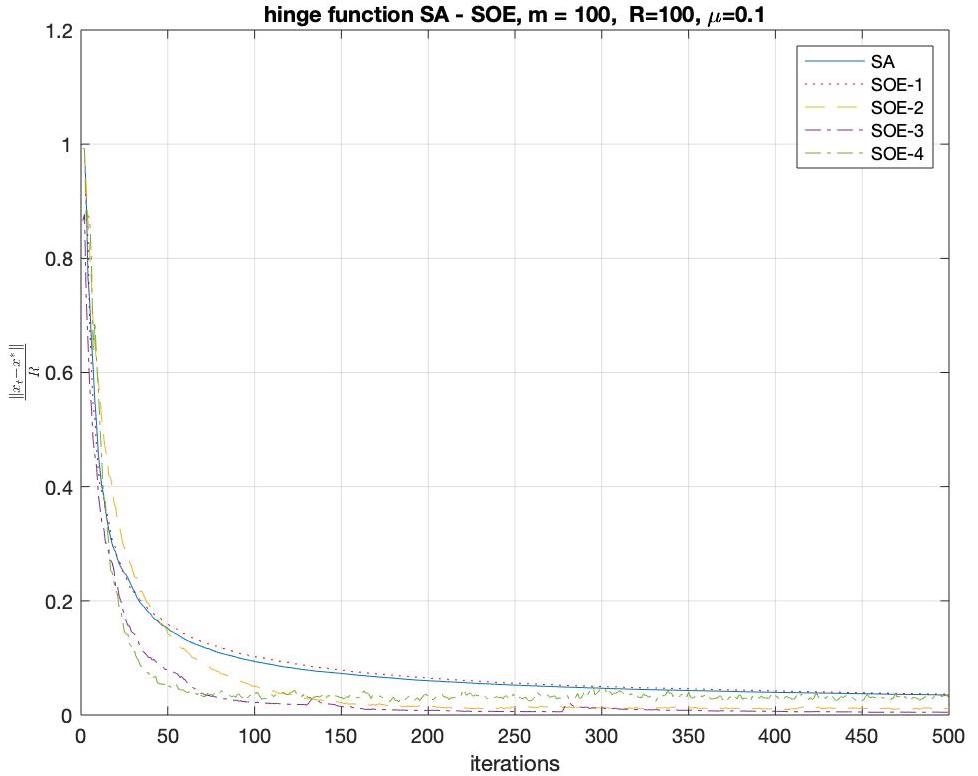}
	\end{minipage}
	\begin{minipage}[t]{0.3\linewidth}
		\centering
		\includegraphics[width=5cm]{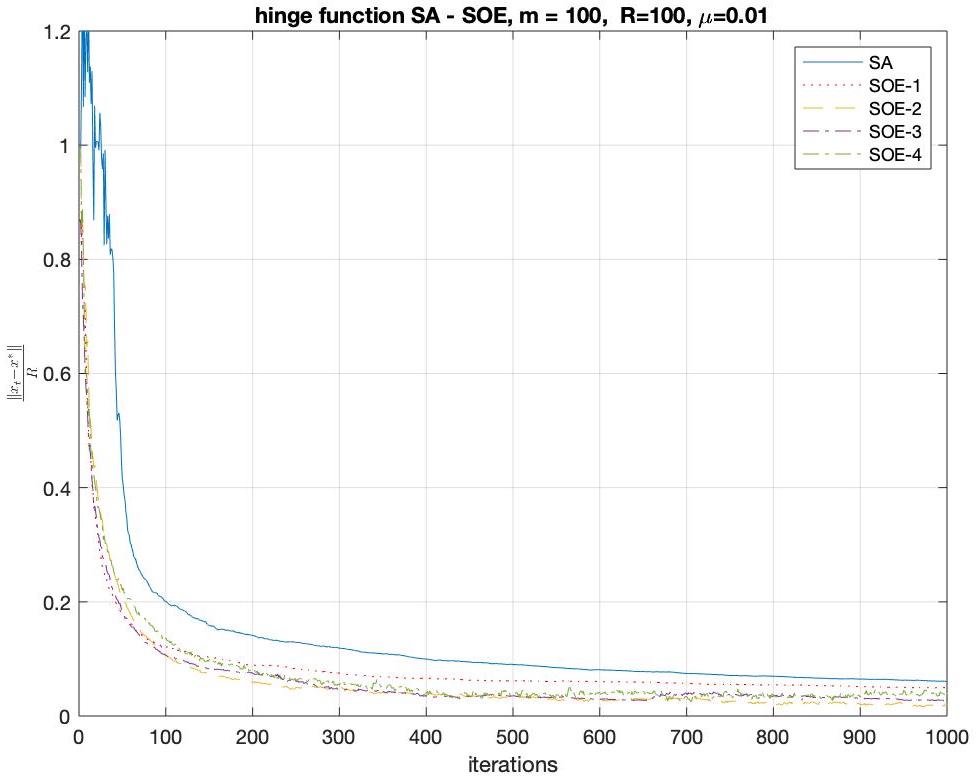}
	\end{minipage}
	\begin{minipage}[t]{0.3\linewidth}
		\centering
		\includegraphics[width=5cm]{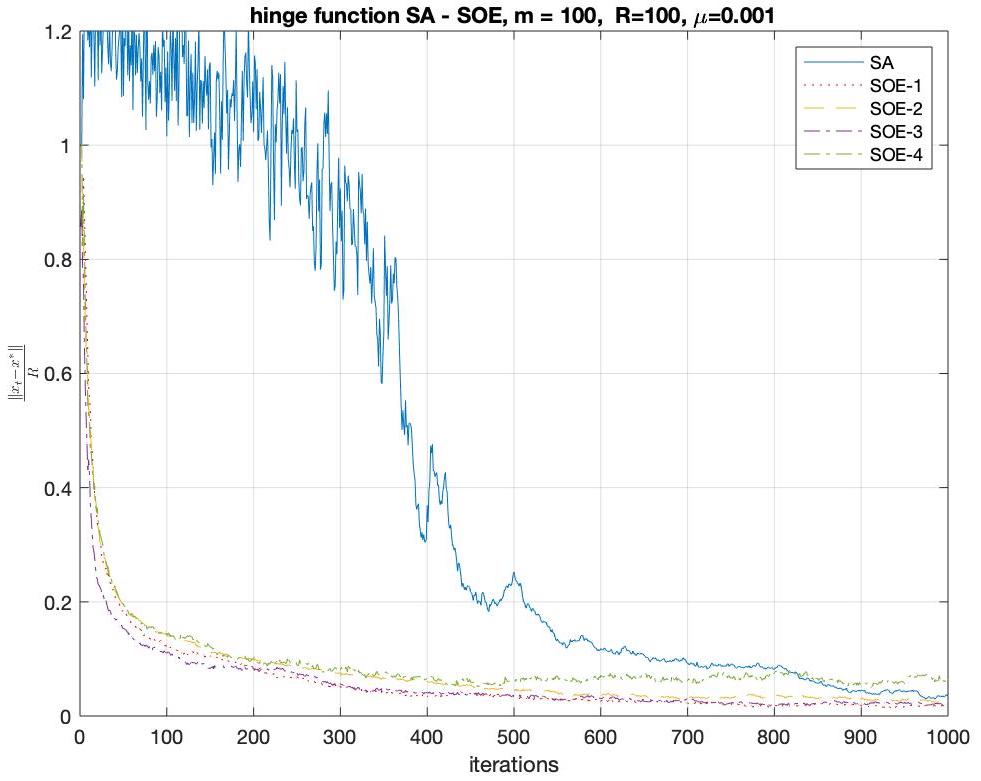}
	\end{minipage}
	\caption{Comparison between SA and SOE-1, SOE-2, SOE-3 and SOE-4 algorithms for the hinge function example. From left to right $d_{-}$ is set to $10^{-1}$,  $10^{-2}$, $ 10^{-3}$ respectively. For each experiment the value of the operator at each iteration was averaged with $m=100$ samples and we set $ \sigma_y = 1$. }
	\label{fig_hinge_function}
\end{figure}
As we can see in figure \ref{fig_hinge_function} above,
the SOE method performs better than SA, with the contrast becoming more pronounced as $\mu$ decreases.
In the next set of experiments (see figure \ref{fig_hinge_function_restart}),
we wanted to manifest the significance of the index resetting scheme. For this purpose we lowered the standard deviation $\sigma_y$ associated to the observation label $y$ and increased the batch size to $ m = 1000$. This step reduces the operator noise and therefore diminishes the size of each epoch, so that we are able to observe the effect of this scheme within the horizon under consideration.
\begin{figure}[H]
	\centering
	\begin{minipage}[t]{0.3\linewidth}
		\centering
		\includegraphics[width=5cm]{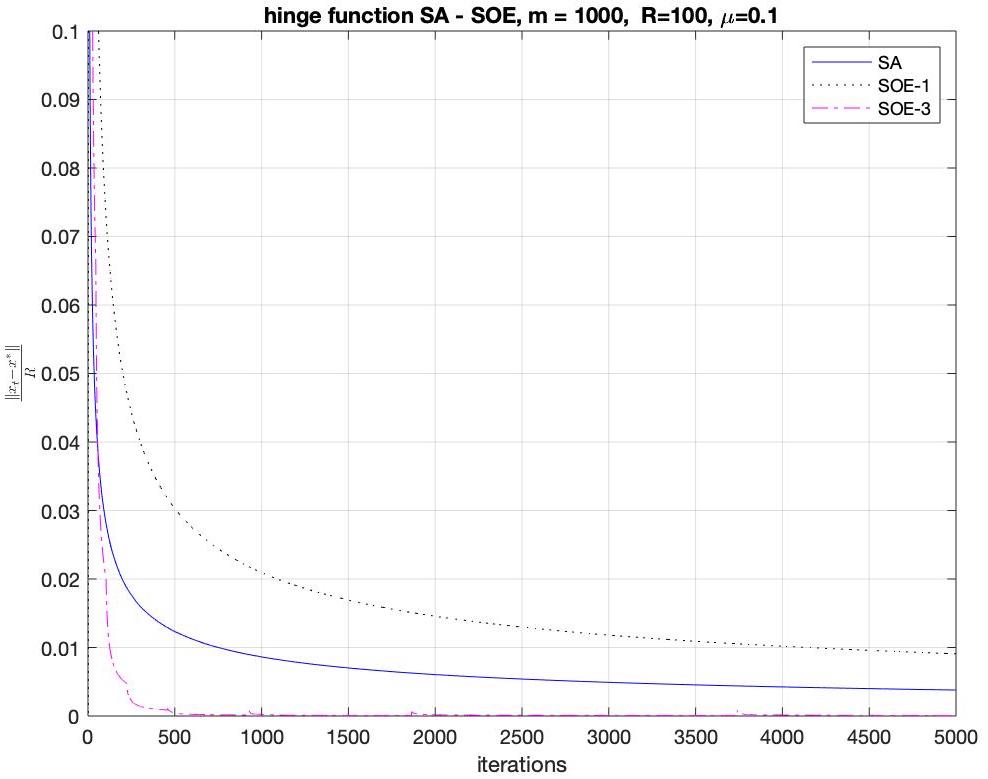}
	\end{minipage}
	\begin{minipage}[t]{0.3\linewidth}
		\centering
		\includegraphics[width=5cm]{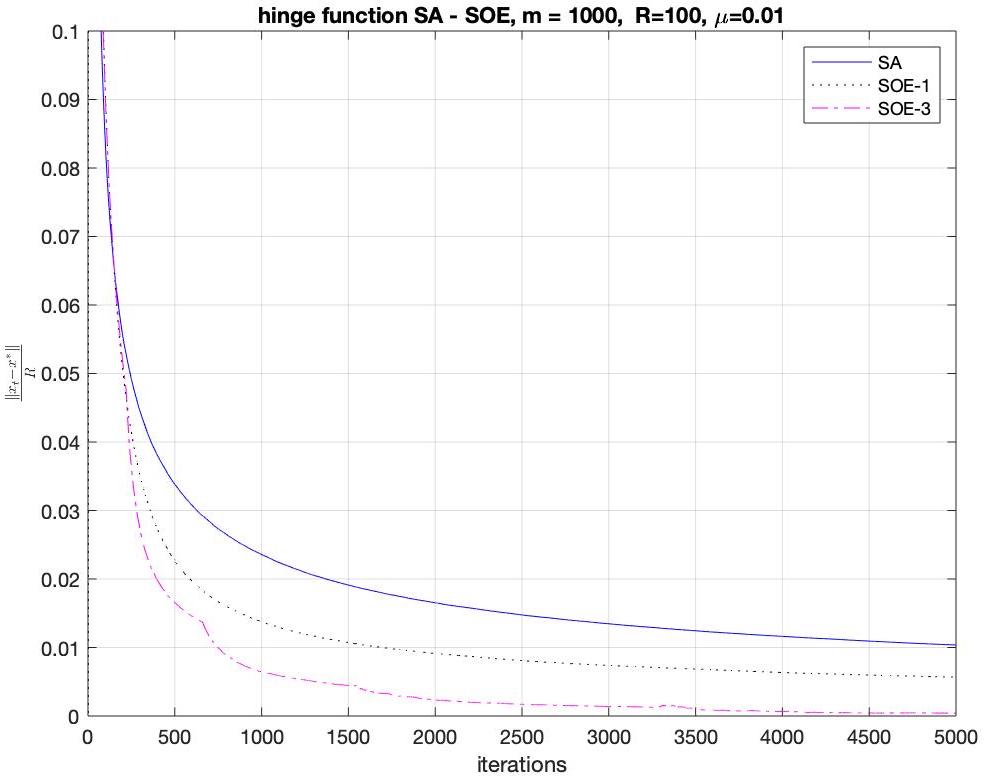}
	\end{minipage}
	\begin{minipage}[t]{0.3\linewidth}
		\centering
		\includegraphics[width=5cm]{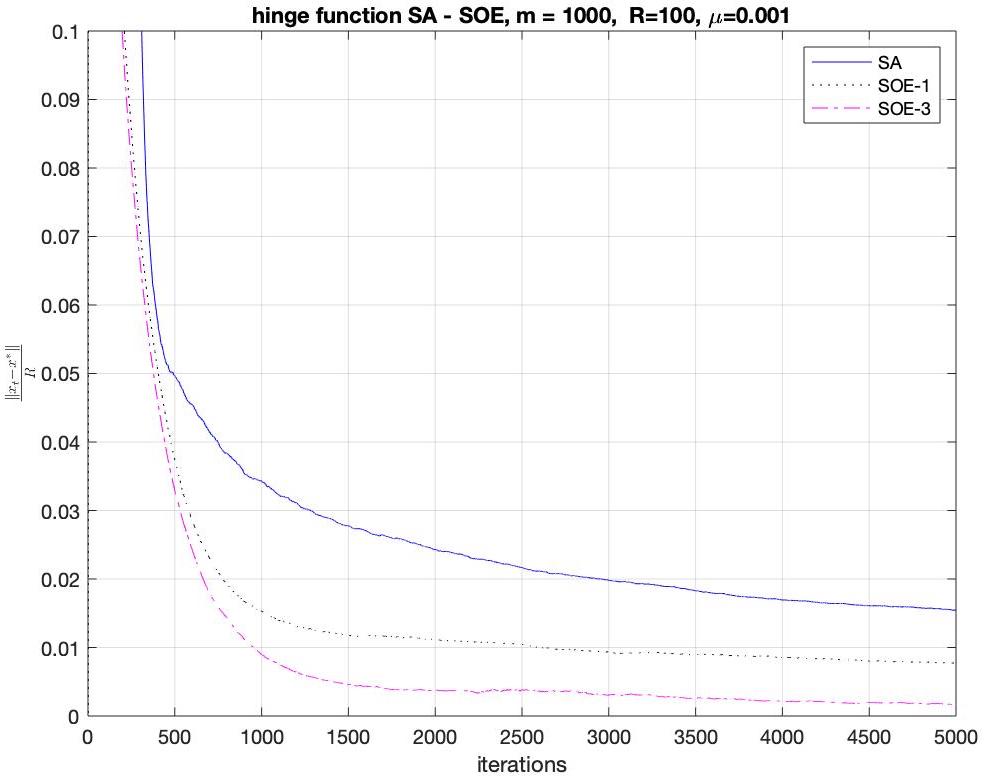}
	\end{minipage}
	\caption{Comparison between SA and SOE-1 and SOE-3 algorithms for the hinge function example. From left to right $d_{-}$ is set to $10^{-1}$,  $10^{-2}$, $10^{-3}$, respectively. For each experiment we use a batch-size $m=1000$ samples and we set  $ \sigma_y = 0.1$. }
	\label{fig_hinge_function_restart}
\end{figure}

\subsubsection{Ramp sigmoid}
In this case the computation of the operator $G_C(x)$ is more involved and we demonstrate it when $A = I_n$.
\begin{eqnarray*}
	G_C(x) & = &  \tfrac{x}{\|x\|} \mathbb{E}[ \zeta   \min\{1, \max\{\zeta \|x\|,0\} \}   ]  
	=   \tfrac{1}{\sqrt{2 \pi}}\tfrac{x}{\|x\|} \bigg( \int_{0}^{\tfrac{1}{\|x\|}} \zeta^2 \|x\|  
	\exp(-\tfrac{\zeta^2}{2}) d\zeta +  \int_{\tfrac{1}{\|x\|}}^{\infty} \zeta   
	\exp(-\tfrac{\zeta^2}{2}) d\zeta
	\bigg) \\
	& = &  \tfrac{1}{\sqrt{2 \pi}}\tfrac{x}{\|x\|}  \bigg( \sqrt{\tfrac{\pi}{2}} ~ \|x \|~   \erf(\tfrac{1}{\sqrt{2} \|  x \|})  \bigg) = \tfrac{1}{2} ~ x ~ \erf(\tfrac{1}{\sqrt{2} \|  x \|}). 
\end{eqnarray*}
\textcolor{black}{The above calculation involves partial integration and employing the definition of the error function,  namely  $ \erf (x)=  \tfrac{2}{\sqrt{\pi}} \int_{0}^x \exp(-t^2) dt$.}
The Jacobian $ \nabla G_C(x) $ is given by 
$
\nabla G_C(x) = \tfrac{1}{2} \erf(\tfrac{1}{\sqrt{2} \|  x \|}) I_n - \tfrac{1}{\sqrt{2 \pi}} \tfrac{\exp(-\tfrac{1}{2 \|x\|^2})}{\|x\|^3}  x x^{\T}.   
$
As such 
$$
L_c = \max_{x \in X} \big\{ \tfrac{1}{2} \erf(\tfrac{1}{\sqrt{2} \|  x \|})\} = \tfrac{1}{2}, ~~~~ 
\mu_C = \min_{x \in X} \big\{ \tfrac{1}{2} \erf(\tfrac{1}{\sqrt{2} \|  x \|})  - \tfrac{1}{\sqrt{2 \pi}} \tfrac{\exp(-\tfrac{1}{2 \|x\|^2})}{\|x\|} \} = 
\tfrac{1}{2} \erf(\tfrac{1}{\sqrt{2} R })  - \tfrac{1}{\sqrt{2 \pi}} \tfrac{\exp(-\tfrac{1}{2 R^2})}{R}.
$$
The modulus of strong monotonicity  is a strictly decreasing function of the radius of the domain and the latter variable is going to be progressively increased in order to achieve more challenging cases from the condition number point of view. As before we choose $n = 100$ and
the vector $x^*$ is created by sampling 
its entries i.i.d. in $[0,1]$ and subsequently normalizing such that $ \| x^* \| = R$.
\begin{figure}[H]
	\centering
	\begin{minipage}[t]{0.3\linewidth}
		\centering
		\includegraphics[width=5cm]{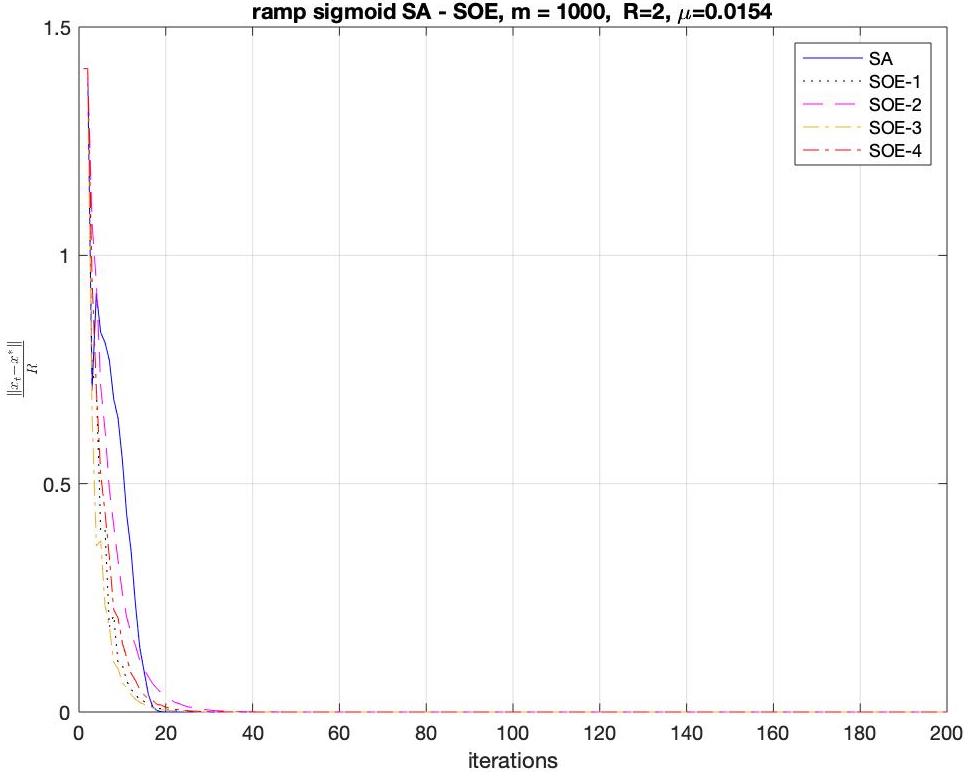}
	\end{minipage}
	\begin{minipage}[t]{0.3\linewidth}
		\centering
		\includegraphics[width=5cm]{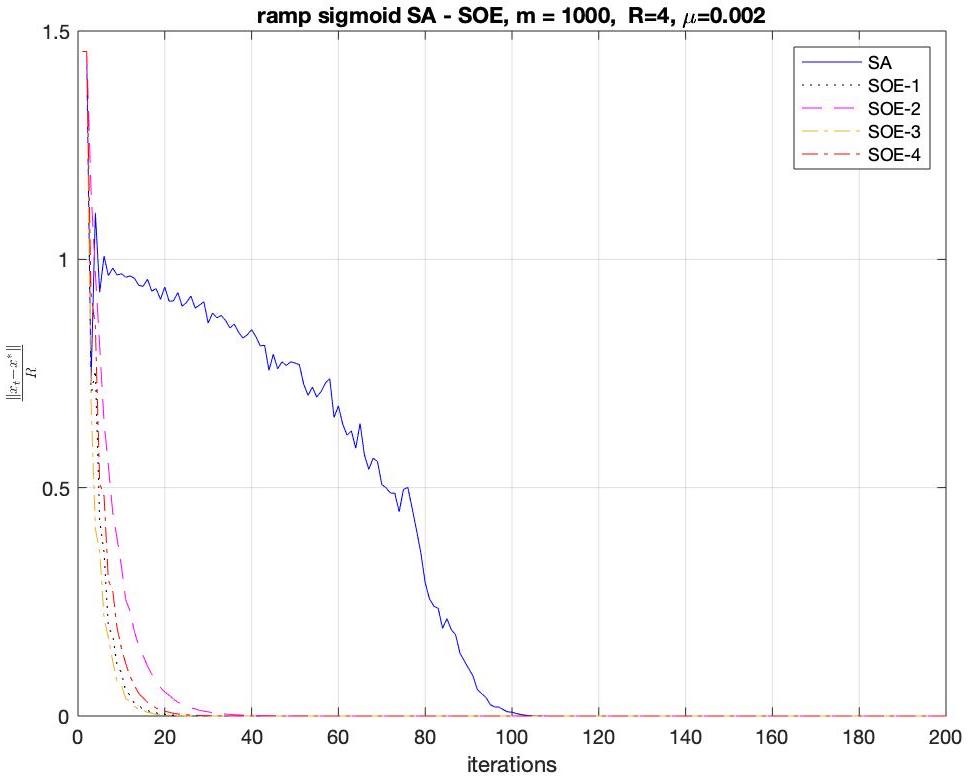}
	\end{minipage}
	\begin{minipage}[t]{0.3\linewidth}
		\centering
		\includegraphics[width=5cm]{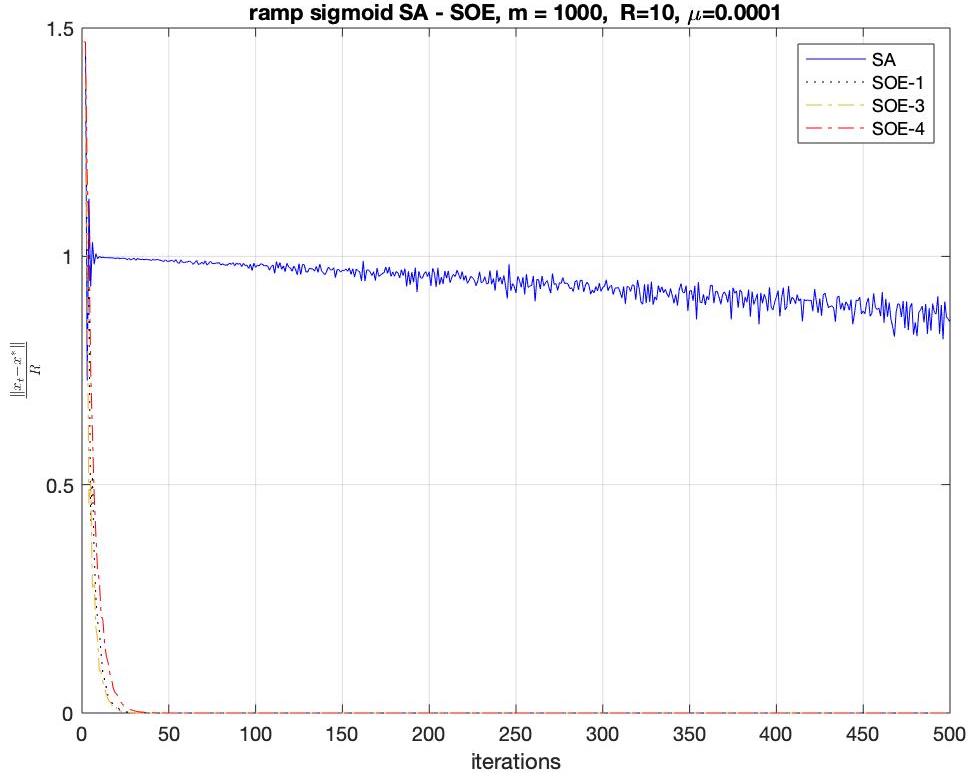}
	\end{minipage}
	\caption{Comparison between SA and SOE-1, SOE-2, SOE-3 and SOE-4 algorithms for the ramp-sigmoid example. From left to right $R$ is set to $2$,  $4$, $10$ respectively. For each experiment we use a batch-size $m=1000$.}
	\label{fig_ramp_sigmoid}
\end{figure}
This is a more challenging example, given the nonlinearity of the operator $F$.  
The results exhibit similar trend as in the hinge function case, in the sense of SOE outperforming SA as the condition number increases,  exasperated by the non-linearity of the operator. 


\section{Concluding remarks} \label{sec_conclusion}
This paper presents a new class of (stochastic) first-order methods obtained by incorporating operator extrapolation into
the gradient (operator) projection methods. We show that the OE method can achieve the optimal
convergence for solving deterministic VIs in a much simpler way than existing approaches. The stochastic counterpart of OE, i.e.,
SOE, achieves the optimal complexity for solving many stochastic VIs, including the stochastic smooth and strongly monotone VIs
 for the first time in the literature. As a smooth optimization method, SOE
allows the application of
mini-batch of samples for variance reduction  and hence facilitates distributed stochastic optimization.
Novel stochastic block operator extrapolation (SBOE) has been proposed,
for the first time in the literature, for solving deterministic VIs with a certain block structure.
Numerical experiments, conducted on a classic traffic assignment problem and
a more recent generalized linear model for signal estimation,
demonstrate the advantages of the proposed algorithms.

\renewcommand\refname{Reference}

\bibliographystyle{siam} 
\bibliography{OE_arXiv.bib}

\end{document}